\documentclass[11pt]{amsart}
\usepackage[utf8]{inputenc}
\usepackage[centertags]{ amsmath}
\usepackage{amssymb,amsthm}
\usepackage{amsfonts}
\usepackage[colorlinks=true,linkcolor=blue,citecolor=blue]{hyperref}
\usepackage{url}
\usepackage{mdframed}
\usepackage{multirow}
\usepackage{csquotes}

\usepackage{enumitem}

\usepackage[colorinlistoftodos,prependcaption,textsize=tiny]{todonotes}

\usepackage{listings}
\usepackage{verbatim}
\usepackage{xcolor,graphicx}

\def\R{\hbox{$\mathbb R$}}

\def\ent{\hbox{{\rm ent}}}
\def\d{\hbox{{\rm d}$s$}}
\def\e{\hbox{{\rm e}}}
\def\p{\hbox{{\rm p}}}

\def\argmin{\hbox{{\rm argmin}}}

\newcommand{\W}{\mathcal W}

\newtheorem{theorem}{Theorem}[section]
\newtheorem{lemma}[theorem]{Lemma}
\newtheorem{proposition}[theorem]{Proposition}

\theoremstyle{definition}
\newtheorem{definition}[theorem]{Definition}

\newtheorem{example}[theorem]{Example}
\newtheorem{remark}[theorem]{Remark}

\title[Proximal Averages for Entropy Minimization]{Proximal Averages for Minimization of Entropy Functionals}

\author[H.H. Bauschke]{Heinz H. Bauschke}

\address[H.H. Bauschke]{Mathematics, University of British Columbia Okanagan, Kelowna, B.C., V1V 1V7, Canada}
\email{{\tt heinz.bauschke@ubc.ca}}

\author[S.B. Lindstrom]{Scott B. Lindstrom}
\address[S.B. Lindstrom]{CARMA, University of Newcastle, Callaghan, Australia, 2308}
\email{\tt scott.lindstrom@uon.edu.au}


\begin{document}

\begin{abstract}
In their 2016 article ``Meetings with Lambert $\W$ and Other Special Functions in Optimization and Analysis'',
 Borwein and Lindstrom considered the minimization of an entropy functional utilizing the weighted average of the negative Boltzmann--Shannon entropy $x\mapsto x\log x -x$ and the energy $x \mapsto x^2/2$; the solution employed the Fenchel-Moreau conjugate of the sum. However, in place of a traditional arithmetic average of two convex functions, it is also, perhaps even more, natural to use their \emph{proximal} average. We explain the advantages and illustrate them by computing the analogous proximal averages for the negative entropy and energy. We use the proximal averages to solve entropy functional minimization problems similar to those considered by Borwein and Lindstrom, illustrating the benefits of using a true homotopy. Through experimentation, we discover computational barriers to obtaining solutions when computing with proximal averages, and we demonstrate a method which appears to remedy them.
\end{abstract}

\maketitle

\section{Introduction}\label{sec:intro}

Computer assisted discovery has changed the way in which research is conducted, both in optimization and elsewhere. As noted by Borwein and Lindstrom \cite{BL2016} in 2016:
\begin{displayquote}[]
In the current mathematical world, it matters
less what you know about a given function than whether your computer
package of choice (say \emph{Maple}, \emph{Mathematica} or \emph{SAGE}) or online source, say
Wikipedia \cite{wikiW} does.
\end{displayquote}
They considered, in particular, that many occurrences of the Lambert $\W$ function in convex analysis may be naturally discovered with the use of the \emph{Symbolic Convex Analysis Tools (SCAT)} package for \emph{Maple} \cite{SCAT,SCATarticle}, notwithstanding one's possible naivety of special functions. The \emph{SCAT} package for \emph{Maple} was created by Borwein and Chris Hamilton;
it grew out of Bauschke and von Mohrenschildt's \emph{Maple} package \emph{fenchel}, which is described in their 2006 article \cite{BvM}. \emph{SCAT} continues to be developed by D.R. Luke and others. F.\ Lauster, D.R.\ Luke, and M.K.\ Tam have recently illuminated symbolic computation with monotone operators \cite{LLT}.

We will continue the work of Borwein and Lindstrom, \cite{BL2016}, by considering \emph{proximal} averages where they originally considered \emph{weighted} averages: for the minimization of entropy functionals. In human-machine collaboration with our computer algebra system (CAS) of choice, \emph{Maple}, we discover closed forms of proximal averages for the Boltzmann--Shannon entropy and energy for specific parameters before conjecturing and eventually proving their general form with the computer assistance. Armed with closed forms, we will consider how the problem of minimizing an entropy functional varies when the weighted average is replaced with a true homotopy.

The structure of this paper is as follows. In Subsection~\ref{ss:Wprelims}, we recall the properties of the Lambert $\W$ function which will prove instrumental in our analysis, and in Subsection~\ref{ss:convexanalysisprelims} we recall preliminaries on convex analysis. In Section~\ref{s:pavs}, we recall the basic properties of the proximal average. In Section~\ref{s:pavswithW} we consider proximal averages which employ $\W$, first the energy and Boltzmann--Shannon entropy in Subsection~\ref{ss:flambda}, and then the energy with the exponential in Subsection~\ref{ss:flambdastar}; the two are related, importantly, through duality. In Section~\ref{s:entropyminimization}, we introduce the problem of minimizing an entropy functional subject to linear constraints, and in Subsection~\ref{ss:entropyexamples} we provide examples. We conclude in Section~\ref{s:conclusion}.

\subsection{Lambert $\W$ preliminaries}\label{ss:Wprelims}

Of particular interest to us is the Lambert $\W$ function, which we take to be the real analytic inverse of $x \mapsto xe^x$. The real inverse is two-valued, and, for the sake of our exposition, we consider $\W$ to always refer to the principal branch, shown in Figure~\ref{fig:Wprincipal}. We will make use of the following elementary identities.

\begin{proposition}
	For any $y \in \mathbb{R}$, the following identities hold:
	\begin{enumerate}[label=\emph{(\roman*)}]
		\item\label{W1} $\W(y)e^{\W(y)} = y$;
		\item\label{W2} $e^{\W(y)}=\frac{y}{\W(y)}$;
		\item\label{W3} $\W(y) = \log \left(\frac{y}{\W(y)}  \right)$;
		\item\label{W4} $\log \left(\W(y) \right) = \log(y)-\W(y)$.
		\item\label{W8} $\log \left(\W(e^y) \right) = y-\W(e^y)$.
	\end{enumerate}
\end{proposition}
\begin{proof}
	\ref{W1}: This is true from the fact that $\W$ is the inverse of $x\mapsto xe^x$.

	\ref{W2}: Divide both sides of \ref{W1} by $\W(y)$.

	\ref{W3}: Take the $\log$ of both sides of \ref{W2}.

	\ref{W4}: Since $\log \left(\frac{y}{\W(y)} \right) = \log(y)-\log(\W(y))$, this follows from \ref{W3}.

	\ref{W8}: Apply \ref{W4}, substituting $e^y$ for $y$.
\end{proof}

An excellent overview of the methods used for symbolic differentiation and anti-differentiation --- and their history --- is given
by R.M.\ Corless, G.H.\ Gonnet, D.E.G.\ Hare, D.J.\ Jeffrey,
and D.E.\ Knuth \cite{CorlessKnuth}. We have, in particular, the following characterization of the derivatives and antiderivative.

\begin{proposition}\label{prop:Wderivative}
	The derivative of $\W$ is given by
	\begin{align*}
	\W'(x) &= \frac{1}{(1+\W(x))\exp(\W(x))}\\
	&= \frac{\W(x)}{x(1+\W(x))}, \quad {\rm if}\; x\neq 0.
	\end{align*}
	Moreover, the $n$th derivative of $\W$ may be characterized as
	\begin{equation*}
	\frac{d^n\W(x)}{dx^n} = \frac{e^{-n\W(x)}p_n(\W(x))}{(1+\W(x))^{2n-1}}\quad {\rm for}\;\;n\geq 1.
	\end{equation*}
	where $p_n(w)$ are polynomials which satisfy the recurrence relation given by
	\begin{equation*}
	p_{n+1}(w) = -\left(nw+3n-1\right)p_n(w)+(1+w)p_n'(w), \quad \text{for}\; n\geq 1.
	\end{equation*}
	For details, see, for example, \cite[Section~3]{CorlessKnuth}.
\end{proposition}

\begin{proposition}
The antiderivative of $\W$ may be characterized as
\begin{align*}
\int \W(x) dx &= \left(\W(x)^2-\W(x)+1\right)e^{\W(x)}+C\\
&=x\left(\W(x)-1+1/\W(x) \right)+C.
\end{align*}
For details, see, for example, \cite[Section~3]{CorlessKnuth}.
\end{proposition}

Using Proposition~\ref{prop:Wderivative}, we also have the following.

\begin{proposition}\label{prop:Wderivatives}
	The following hold:
	\begin{enumerate}[label=\emph{(\roman*)}]
		\item\label{W5} $\frac{d}{dx} \W(e^x)=\frac{\W(e^x)}{1+\W(e^x)}$.
		\item\label{W6}  $ \frac{d}{dx}\left(\W(e^x)+\frac{1}{2}\W(e^x)^2\right)=\W(e^x)$;
		\item\label{W7} $\frac{d}{dx}e^{\W(x)} = \frac{1}{1+\W(x)}$;
	\end{enumerate}
\end{proposition}
\begin{proof}
	\ref{W5}: Apply the chain rule along with the identity from Proposition~\ref{prop:Wderivative} to differentiate $\W(e^x)$.

	\ref{W6}: Apply the chain rule along with the identity from Proposition~\ref{prop:Wderivative} to differentiate $\left(\W(e^x)+\frac{1}{2}\W(e^x)^2\right)$.

	\ref{W7}: Apply the chain rule along with the identity from Proposition~\ref{prop:Wderivative} to differentiate $e^{\W(x)}$.
\end{proof}

\subsection{Preliminaries on Convex Analysis}\label{ss:convexanalysisprelims}

Throughout, $X$ is a Hilbert space.

\begin{definition}\label{F}As in \cite{BLT2007}, we will work with the following set of functions:
	\begin{equation*}
	\mathcal{F}:= \left \{f: X \rightarrow \left]-\infty,\infty \right]\; |\; f \text{ is convex, lower semicontinuous, and proper}  \right \}.
	\end{equation*}
\end{definition}

\begin{figure}
	\begin{center}
		\includegraphics[width=.4\textwidth]{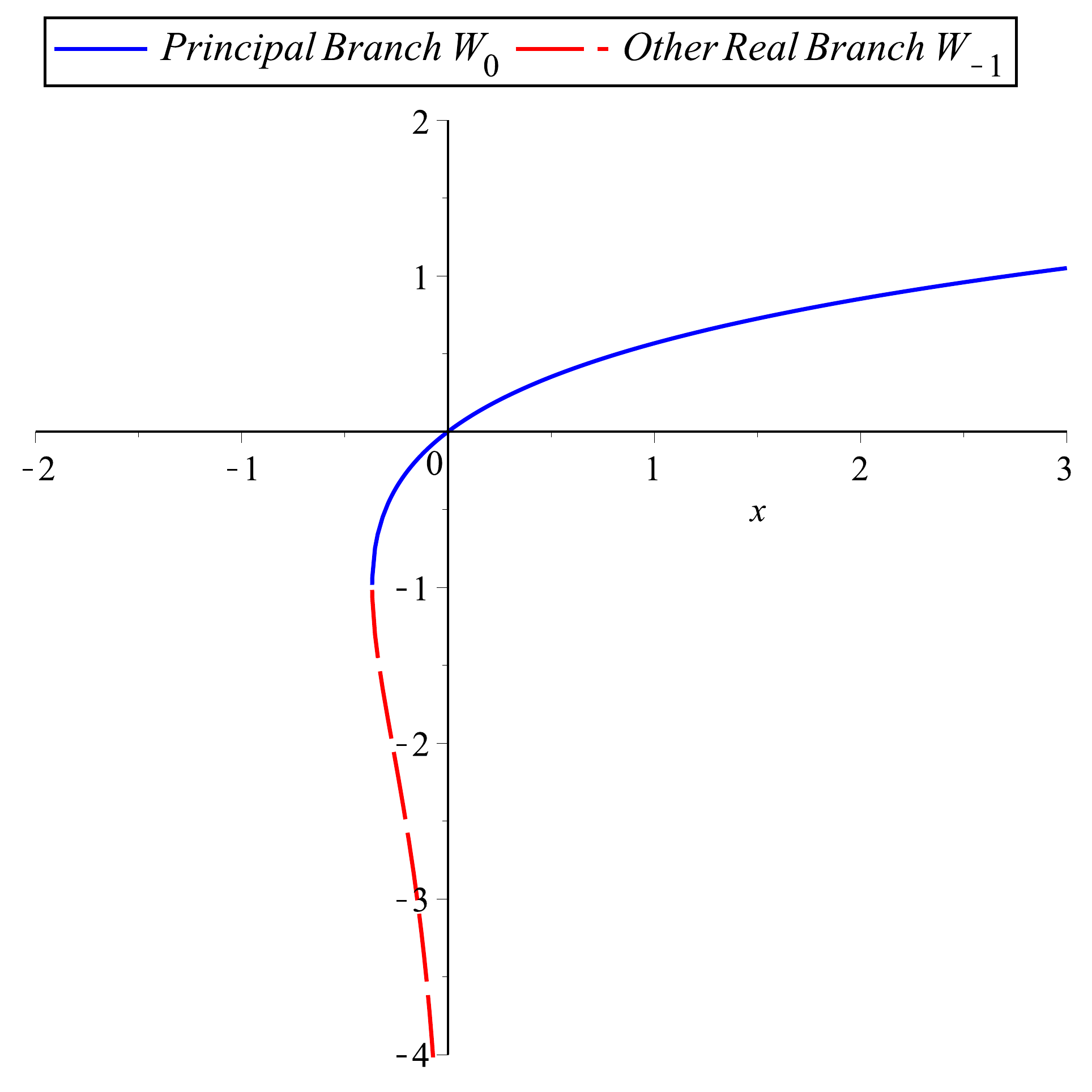}
	\end{center}
	\label{fig:Wprincipal}
	\caption{The two real branches of Lambert $\W$.}
\end{figure}

\begin{definition}[Fenchel Conjugate]The Fenchel conjugate $f^*$ of a function $f:X \rightarrow [-\infty,\infty]$ is defined as follows:
	\begin{align*}
	f^*: X^* &\rightarrow [-\infty,\infty]\\
	f^*:x &\mapsto \underset{y \in X}{\sup} \left \{ \langle x,y \rangle -f(y) \right \}.
	\end{align*}
	This is also often referred to as a \emph{convex conjugate} or \emph{Fenchel-Moreau conjugate}.
\end{definition}

The function $f^*$ is always convex, i.e. its epigraph is convex. Moreover, we have the following.

\begin{proposition}{\cite[Proposition 16.4]{BC}}
Let $f \in \mathcal{F}$ and $x \in {\rm dom}\, \partial f$. Then $f^{**}=f$ and $\partial f^{**}(x)=\partial f(x)$.
\end{proposition}

\begin{definition}[Argmin operator $\p$ for Fenchel conjugates]Let $f$ be a proper convex function and $f^*$ its conjugate. We define $\p_{f}$ to be a selection operator satisfying
	\begin{equation}
	\p_f(x) \in \underset{y \in X}{\argmin} \left \{ \langle x,y \rangle -f(y) \right \}.
	\end{equation}
	so that we may express the closed form for $f^*$ as
	\begin{equation}
	f^*(x) = \underset{y \in X}{\sup} \{ \langle x,y \rangle -f(y)\} = \langle x, \p_f(x)\rangle - f(\p_f(x)).
	\end{equation}
\end{definition}

\section{Proximal Averages}\label{s:pavs}

The systematic investigation of the proximal average started in 2008
\cite{BLT2007}, relying crucially on an important result of
Bauschke, E.\ Matou\u{s}kov\'{a}, and S.\ Reich \cite[Theorem~6.1]{BMR2004}.
Research on the topic continues to grow.
One noteworthy recent application is Y.L.\ Yu's 2013
employment of the proximal average to analyse a novel
proximal gradient algorithm \cite{Yu}.

\begin{definition}[Proximal Average]\label{def:pav}The proximal average operator is
	\begin{align*}
	\mathcal{P}: \mathcal{F} \times \left[0,1\right] \times \mathcal{F} &\rightarrow \left \{f | f: X \rightarrow \left[ -\infty,+\infty \right] \right \}\\
	\left(f_0,\lambda,f_1\right)&\mapsto \resizebox{.7\textwidth}{!}{$%
	\left((1-\lambda)\left(f_0 + \frac{1}{2}\|\cdot \|^2\right)^* + \lambda \left(f_1+\frac{1}{2}\| \cdot \|^2\right)^*  \right)^* - \frac{1}{2}\|\cdot \|^2.$%
	}
	\end{align*}
	See, for example, \cite[Definition~4.1]{BLT2007}.
\end{definition}

\begin{remark}[Symmetric and convex properties of proximal averages]Let $f_0,f_1 \in \mathcal{F}$ and $\lambda \in [0,1]$. Then we have that
	\begin{equation}
	\mathcal{P}(f_0,0,f_1)=f_0, \quad \mathcal{P}(f_0,1,f_1)=f_1, \quad \text{and} \; \mathcal{P}(f_0,\lambda,f_1)=\mathcal{P}(f_1,1-\lambda,f_0).
	\end{equation}
	We also have that $\mathcal{P}(f_0,\lambda,f_1)$ is convex.
	See, for example, \cite[Proposition~4.2]{BLT2007}.
\end{remark}

\begin{remark}[Conjugacy of proximal averages]\label{rem:conjugacy}When $f_0,f_1 \in \mathcal{F}$ and $\lambda \in [0,1]$ we have that
	\begin{equation}\label{conjugacy}
	\left(\mathcal{P}(f_0,\lambda,f_1) \right)^* = \mathcal{P}(f_{0}^{*},\lambda,f_1^*).
	\end{equation}
	See, for example, \cite[Theorem 4.3]{BLT2007} or \cite[Theorem 6.1]{BMR2004}.
\end{remark}

\begin{definition}[Simplified notation for proximal averages]We will follow a convenient convention from \cite{BLT2007}. Let $f_0,f_1 \in \mathcal{F}$ and $\lambda \in [0,1]$. Let
	\begin{equation*}
	f_\lambda := \mathcal{P}(f_0,\lambda,f_1) \quad \text{and}\quad f_\lambda^* :=\mathcal{P}(f_0^*,\lambda,f_1^*).
	\end{equation*}
	From Remark~\ref{rem:conjugacy} we have that $(f_\lambda)^* = (f^*)_\lambda$, which shows that $f_\lambda^*$ is not ambiguous.
\end{definition}

\begin{definition}[epi-convergence and epi-topology]\label{def:epiconvergence}
	Let $f$ and $\left(f_n \right)_{n \in \mathbb{N}}$ be functions from $X$ to $\left]-\infty,+\infty \right]$. Then $(f_n)_{n \in \mathbb{N}}$ epi-converges to $f$ if for every $x \in X$ the following hold
	\begin{enumerate}[label=(\roman*)]
		\item For every sequence $(x_n)_{n \in \mathbb{N}}$ in $X$ converging to $x$, one has $f(x) \leq \liminf f_n(x_n)$.
		\item There exists a sequence $(y_n)_{n \in \mathbb{N}}$ in $X$ converging to $x$ such that $\limsup f_n (y_n) \leq f(x)$.
	\end{enumerate}
	in which case we write $f_n \overset{\scriptsize \e}{\rightarrow} f$. The \emph{epi-topology} is the topology induced by epi-convergence. See, for example, \cite[Definition 5.1]{BLT2007}. For greater detail, see \cite{RW}.
\end{definition}

\begin{remark}[Continuity of $\mathcal{P}$]\label{rem:continuity} Suppose that $\mathcal{F}$ is equipped with the epi-topology. Then the proximal average operator $\mathcal{P}: \mathcal{F} \times [0,1] \times \mathcal{F} \rightarrow \mathcal{F}$ is continuous. In other words, where $(f_n)_{n\in \mathbb{N}},(g_n)_{n\in \mathbb{N}}$ are sequences in $\mathcal{F}$ and $(\lambda_n)_{n \in \mathbb{N}}$ is a sequence in $[0,1]$ such that $f_n \overset{\scriptsize \e}{\rightarrow}f,g_n \overset{\scriptsize \e}{\rightarrow}g$, and $\lambda_n \rightarrow \lambda$, then we have that:
	\begin{equation}
	\mathcal{P}(f_n,\lambda_n,g_n) \overset{\scriptsize \e}{\rightarrow} \mathcal{P}\left(f,\lambda,g \right) \quad \text{as}\quad n \rightarrow \infty
	\end{equation}
For a proof, see, for example, \cite[Theorem 5.4]{BLT2007}.
\end{remark}

\section{Proximal Averages Employing Lambert $\W$}\label{s:pavswithW}

\begin{definition}We define the negative Boltzmann--Shannon entropy as follows:
	\begin{equation}
	\ent : \R \rightarrow \R \cup \{\infty \}: \quad x \mapsto \begin{cases}
	x\log x -x & x\in \left]0,\infty \right]\\
	0 & x=0\\
	\infty & \text{otherwise}
	\end{cases}
	\end{equation}
\end{definition}

In \cite{BL2016} the authors considered the average (\emph{not} the proximal average) given by
\begin{equation}\label{originalaverage}
f_t(x)=(1-t)\ent(x)+t\frac{x^2}{2}
\end{equation}
for $0\leq t \leq 1$ so that $f_0$ is the Boltzmann--Shannon entropy and $f_1$ is the energy. For clarity, we will refer to such an average as a \emph{weighted} average, in order to distinguish it from the \emph{proximal} average, and we will consistently use $t$ for the former and $\lambda$ for the latter.

The Borwein and Lindstrom then obtained the conjugate as follows:
\begin{equation}\label{originalconjugate}f_t^*(y) =\dfrac{(1-t)^2}{2t}\left(\W\left(\frac{t}{1-t}e^{\dfrac{y}{1-t}}\right)+2\right)\W\left(\frac{t}{1-t}e^{\dfrac{y}{1-t}}\right).\end{equation}

\begin{remark}[Limiting Cases for Weighted Average]\label{originallimitingcases}In \eqref{originalaverage}, if one considers the \emph{limit} for $f_t$ as $t\rightarrow 1$ we obtain the \emph{positive energy} which is infinite at negative points. In the limit as $t\rightarrow 0$ we recover $\ent(x)$. For its conjugate in $f_t^*$ in \eqref{originalconjugate}, if one considers the limit for at $t=0$ we recover $\exp(x)$ which is the conjugate of $\ent(x)$. In the limit at $t=1$ we obtain
	\begin{equation*}
	x \mapsto \begin{cases}
	\frac{x^2}{2} & x>0\\
	0 & \text{otherwise}
	\end{cases}.
	\end{equation*}
	We would expect this, given that $\frac{(\cdot)^2}{2}$ is self-conjugate while $\ent(x)$ is infinite for $x<0$. Notice, however, that $f_1^* = \frac{1}{2}|\cdot|^2$, and so we do not reobtain $f_1^*$ in the limiting case as $t\rightarrow 1$.
\end{remark}

Both $f_t$ and $f_t^*$ may be seen in Figure~\ref{fig:original_average}.

\begin{figure}
	\begin{center}
		\includegraphics[height=5cm]{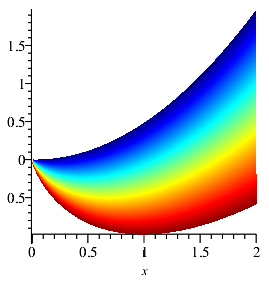}
		\includegraphics[height=5cm]{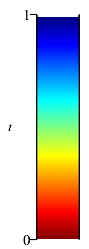}
		\includegraphics[height=5cm]{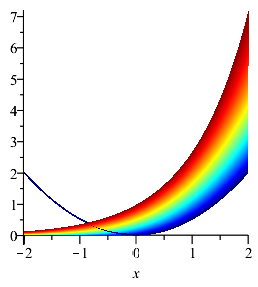}
	\end{center}
	\caption{$f_t$ from \eqref{originalaverage} (left) and $f_t^*$ from \eqref{originalconjugate} (right).} \label{fig:original_average}
\end{figure}

However, instead of \eqref{originalaverage}, it is more natural to consider $f_\lambda = \mathcal{P}(\ent,\lambda,\frac{(\cdot)^2}{2})$. We will compute $f_\lambda$ and its conjugate $f_\lambda^*$ which is the natural
analogue to~\eqref{originalconjugate}.

\subsection{Form and proof for $f_\lambda$}\label{ss:flambda}

Throughout the following $\lambda \in \left]0,1\right[$.

\begin{lemma}\label{lem:leftandrightconjs}Let $f_0 := \ent$ and $f_1 := \frac12 (\cdot)^2$. Then we have the following:
	\begin{enumerate}[label=\emph{(\roman*)}]
	\item \label{leftconj} $\left(f_0 + \frac{1}{2}(\cdot)^2 \right)^* = \frac{1}{2}\W(e^x)\left(\W(e^x)+2 \right)$
	\item \label{rightconj} $\left(f_1 + \frac{1}{2}(\cdot)^2 \right)^* = \frac{1}{4}(\cdot)^2$
	\end{enumerate}
\end{lemma}
\begin{proof}
	\ref{leftconj}: By definition,
	\begin{equation*}
	\left(f_0 + \frac{1}{2}(\cdot)^2 \right)^*(x) = \underset{y \in {\small \R}}{\sup} \left \{xy-f_0(y)-\frac12 y^2 \right \}.
	\end{equation*}
	Differentiating the inner term with respect to $y$ and setting equal to zero, we have that the supremum is obtained when $y$ satisfies $x-\log(y)-y=0$. Solving for $y$, we obtain
	\begin{align*}
	\log(y)&=x-y,\\
	\text{and so}\quad y&=\frac{e^x}{e^y},\\
	\text{which simplifies to}\quad ye^y &=e^x,
	\end{align*}
	and so $y=\W(e^x)$. Substituting this value back into $xy-f_0(y)-\frac12 y^2$, we have that
	\begin{equation*}
	\left(f_0 + \frac{1}{2}(\cdot)^2 \right)^*(x) = x\W(e^x) -\W(e^x)\log\left(\W(e^x) \right)+\W(e^x) - \frac{1}{2}\W(e^x)^2.
	\end{equation*}
	Factoring and employing the fact that $(\log \circ \W)(z) = \log(z)-\W(z)$ we obtain
	\begin{align*}
	\left(f_0 + \frac{1}{2}(\cdot)^2 \right)^*(x) =\frac{1}{2}\W(e^x)\left(2x + \W(e^x) - 2\log(e^x)+2\right)
	\end{align*}
	Because $x$ is real, the right hand side further simplifies to $\frac{1}{2}\W(e^x)\left(\W(e^x)+2 \right)$, completing the proof of \ref{leftconj}.\newline
	\ref{rightconj}: This is a well known result and may be obtained by simple arithmetic.
\end{proof}

Note that we may recognize the term $\frac{1}{2}\W(e^x)(\W(e^x)+2)$ as an antiderivative of $\W(e^x)$ (see Proposition~\ref{prop:Wderivatives}), a fact we will exploit in the following lemma.

\begin{lemma}\label{lem:prox}
	Let $\varphi$ be defined as follows
	\begin{align*}
	\varphi &:= \left((1-\lambda)\left(\frac12 \W(e^{(\cdot)})(\W(e^{(\cdot)})+2)\right) + \lambda\left(\frac14(\cdot)^2 \right) \right)^*.
	\end{align*}
	Then it holds that
	\begin{equation}\label{term0}
	\p_\varphi (x) = - \frac{(\frac{2}{\lambda}-2)\W \left(\left(\frac{2}{\lambda}-1 \right)e^{\frac{2x}{\lambda}}   \right)}{\frac{2}{\lambda}-1}+\frac{2x}{\lambda}
	\end{equation}
	so that we may explicitly write
	\begin{equation*}
	\varphi(x) = x \p_\varphi(x)- (1-\lambda)\left(\frac{1}{2}\W(e^{\p_\varphi(x)})(\W(e^{\p_\varphi(x)})+2) \right) -\frac{\lambda}{4}\p_\varphi(x)^2.
	\end{equation*}
\end{lemma}

\begin{proof}
	By definition,
	\begin{equation*}
	\varphi(x) = \underset{y \in \R}{\sup}\{xy- (1-\lambda)\left(\frac{1}{2}\W(e^y)(\W(e^y)+2) \right) -\frac{\lambda}{4} y^2 \}.
	\end{equation*}
	Differentiating the inner term with respect to $y$ and setting equal to zero, we obtain
	\begin{equation}\label{exprmanual}
	x-(1-\lambda)\W(e^y)-\frac{\lambda}{2}y=0.
	\end{equation}
	We will show that \eqref{exprmanual} is true if $y=\p_\varphi(x)$. First we will rewrite \eqref{exprmanual} using the fact that $\W(a) = b$ if and only if $be^b = a$, which allows us to remove the $\W(e^y)$ term as follows:
	\begin{align*}
	\W(e^y)&=\frac{x-\frac{\lambda}{2}y}{1-\lambda},\\
	\left(\frac{x-\frac{\lambda}{2}y}{1-\lambda} \right)e^{\left(\frac{x-\frac{\lambda}{2}y}{1-\lambda} \right)} &= e^y,\\
	\left(x-\frac{\lambda}{2}y \right)e^{\frac{y\lambda-2x}{2(\lambda-1)}}&=(1-\lambda)e^y.
	\end{align*}
	This is equivalent to the form returned by \emph{Maple},
	\begin{equation}\label{functiontosolve}
	e^{\frac{y\lambda -2x}{2(\lambda-1)}}y\lambda -2e^y \lambda -2xe^{\frac{y\lambda-2x}{2(\lambda-1)}}+2e^y =0,
	\end{equation}
	and so again naivety need not inhibit the discovery. We will use \emph{Maple}'s form. We need only to show that
	\begin{equation}\label{WTS}
	e^{\frac{\p_\varphi(x)\lambda -2x}{2(\lambda-1)}}\p(x)\lambda -2e^{\p_\varphi(x)} \lambda -2xe^{\frac{\p_\varphi(x)\lambda-2x}{2(\lambda-1)}}+2e^{\p_\varphi(x)} =0.
	\end{equation}
	First consider the term $e^{\p_\varphi(x)}$. Since for any $a,b,z$ we have that
	$$e^{aW(z)+b}=(e^{W(z)})^a e^b = \left(\frac{z}{W(z)}\right)^a e^b, $$
	we may let
	\begin{equation}
	a:=-\left(\frac{\frac{2}{\lambda}-2}{\frac{2}{\lambda}-1}\right), \quad b:=\frac{2x}{\lambda}, \quad z:=\left(\frac{2}{\lambda}-1\right)e^{\frac{2x}{\lambda}},
	\end{equation}
	and thusly rewrite
	\begin{equation}\label{term1}
	e^{\p_\varphi(x)}= \left(\frac{\left(\frac{2}{\lambda}-1\right)e^{\frac{2x}{\lambda}}}{\W \left(\left(\frac{2}{\lambda}-1\right)e^{\frac{2x}{\lambda}} \right)}\right)^{-\frac{\frac{2}{\lambda}-2}{\frac{2}{\lambda}-1}} e^{\frac{2x}{\lambda}}.
	\end{equation}
	Next consider the term $e^{\frac{\p_\varphi(x)\lambda  -2x}{2(\lambda-1)}}$. Using \eqref{term1}, we may rewrite it thusly:
	\begin{align}
	e^{\frac{\p_\varphi(x)\lambda  -2x}{2(\lambda-1)}} &= e^{\frac{-2x}{2(\lambda-1)}}\left(e^{\p_\varphi(x)}\right)^{\frac{\lambda}{2(\lambda-1)}}\nonumber \\
	&=e^{\frac{-2x}{2(\lambda-1)}}\left(\left(\frac{\left(\frac{2}{\lambda}-1\right)e^{\frac{2x}{\lambda}}}{\W \left(\left(\frac{2}{\lambda}-1\right)e^{\frac{2x}{\lambda}} \right)}\right)^{-\frac{\frac{2}{\lambda}-2}{\frac{2}{\lambda}-1}} e^{\frac{2x}{\lambda}}\right)^{\frac{\lambda}{2(\lambda-1)}}\nonumber \\
	&= e^{\frac{-2x}{2(\lambda-1)}}\left(\frac{\left(\frac{2}{\lambda}-1\right)e^{\frac{2x}{\lambda}}}{\W \left(\left(\frac{2}{\lambda}-1\right)e^{\frac{2x}{\lambda}} \right)}\right)^{-\frac{\lambda}{\lambda-2}}e^{\frac{2x}{2(\lambda-1)}}\nonumber \\
	&=\left(\frac{\left(\frac{2}{\lambda}-1\right)e^{\frac{2x}{\lambda}}}{\W \left(\left(\frac{2}{\lambda}-1\right)e^{\frac{2x}{\lambda}} \right)}\right)^{-\frac{\lambda}{\lambda-2}}\nonumber \\
	&=\frac{\left(\frac{2}{\lambda}-1\right)^{-\frac{\lambda}{\lambda-2}}e^{-\frac{2x}{\lambda-2}}}{\W \left(\left(\frac{2}{\lambda}-1\right)e^{\frac{2x}{\lambda}} \right)^{-\frac{\lambda}{\lambda-2}}}.\label{term2}
	\end{align}
	Next consider the term $e^{\frac{\p_\varphi(x)\lambda -2x}{2(\lambda-1)}}\p_\varphi(x)\lambda$. Using \eqref{term0} and \eqref{term2}, we may rewrite it as follows:
	\begin{align}
	e^{\frac{\p_\varphi(x)\lambda -2x}{2(\lambda-1)}}\p_\varphi(x)\lambda =&\resizebox{.75\textwidth}{!}{$%
		 \frac{\left(\frac{2}{\lambda}-1\right)^{\frac{-\lambda}{\lambda-2}}e^{\frac{-2x}{\lambda-2}}}{\W \left(\left(\frac{2}{\lambda}-1\right)e^{\frac{2x}{\lambda}} \right)^{\frac{-\lambda}{\lambda-2}}} \left(- \frac{(\frac{2}{\lambda}-2)\W \left(\left(\frac{2}{\lambda}-1 \right)e^{\frac{2x}{\lambda}}   \right)}{\frac{2}{\lambda}-1}+\frac{2x}{\lambda}\right) \lambda \nonumber$%
		}\\
	=&\resizebox{.75\textwidth}{!}{$%
		-\left(\frac{2}{\lambda}-1\right)^{\frac{-\lambda}{\lambda-2}-1}\left( \frac{2}{\lambda}-2\right)e^{-\frac{2x}{\lambda-2}}\W \left(\left(\frac{2}{\lambda}-1  \right)e^{\frac{2x}{\lambda}}  \right)^{\left(1+\frac{\lambda}{\lambda-2}\right)}\lambda \nonumber$%
	}\\
	+&\frac{2x\left(\frac{2}{\lambda}-1\right)^{-\frac{\lambda}{\lambda-2}}e^{-\frac{2x}{\lambda-2}}}{\W \left(\left(\frac{2}{\lambda}-1\right)e^{\frac{2x}{\lambda}} \right)^{-\frac{\lambda}{\lambda-2}}}.\label{term3}
	\end{align}
	Using \eqref{term1}, \eqref{term2}, and \eqref{term3}, we may have that the statement \eqref{WTS} --- which we want to show --- is equivalent to:
	\begin{align*}
	0&=-\left(\frac{2}{\lambda}-1\right)^{\frac{-\lambda}{\lambda-2}-1}\left( \frac{2}{\lambda}-2\right)e^{-\frac{2x}{\lambda-2}}\W \left(\left(\frac{2}{\lambda}-1  \right)e^{\frac{2x}{\lambda}}  \right)^{\left(1+\frac{\lambda}{\lambda-2}\right)}\lambda\\
	&+\frac{2x\left(\frac{2}{\lambda}-1\right)^{-\frac{\lambda}{\lambda-2}}e^{-\frac{2x}{\lambda-2}}}{\W \left(\left(\frac{2}{\lambda}-1\right)e^{\frac{2x}{\lambda}} \right)^{-\frac{\lambda}{\lambda-2}}}+2(1-\lambda)\left(\frac{\left(\frac{2}{\lambda}-1\right)e^{\frac{2x}{\lambda}}}{\W \left(\left(\frac{2}{\lambda}-1\right)e^{\frac{2x}{\lambda}} \right)}\right)^{-\frac{\frac{2}{\lambda}-2}{\frac{2}{\lambda}-1}} e^{\frac{2x}{\lambda}}\\ &-2x\frac{\left(\frac{2}{\lambda}-1\right)^{-\frac{\lambda}{\lambda-2}}e^{-\frac{2x}{\lambda-2}}}{\W \left(\left(\frac{2}{\lambda}-1\right)e^{\frac{2x}{\lambda}} \right)^{-\frac{\lambda}{\lambda-2}}}.
	\end{align*}
	Now the positive and negative terms of the form
	\begin{equation*}
	2x\frac{\left(\frac{2}{\lambda}-1\right)^{-\frac{\lambda}{\lambda-2}}e^{-\frac{2x}{\lambda-2}}}{\W \left(\left(\frac{2}{\lambda}-1\right)e^{\frac{2x}{\lambda}} \right)^{-\frac{\lambda}{\lambda-2}}}
	\end{equation*}
	cancel each other out, leaving us with
	\begin{align}
	0&=-\left(\frac{2}{\lambda}-1\right)^{\frac{-\lambda}{\lambda-2}-1}\left( \frac{2}{\lambda}-2\right)e^{-\frac{2x}{\lambda-2}}\W \left(\left(\frac{2}{\lambda}-1  \right)e^{\frac{2x}{\lambda}}  \right)^{\left(1+\frac{\lambda}{\lambda-2}\right)}\lambda \nonumber \\
	&+2(1-\lambda)\left(\frac{\left(\frac{2}{\lambda}-1\right)e^{\frac{2x}{\lambda}}}{\W \left(\left(\frac{2}{\lambda}-1\right)e^{\frac{2x}{\lambda}} \right)}\right)^{-\frac{\frac{2}{\lambda}-2}{\frac{2}{\lambda}-1}} e^{\frac{2x}{\lambda}}. \nonumber
	\end{align}
	Rewriting and simplifying, we obtain
	\begin{align}
	0&=-\left(\frac{2}{\lambda}-1\right)^{\frac{2(1-\lambda)}{\lambda-2}}2\left( 1-\lambda\right)e^{-\frac{2x}{\lambda-2}}\W \left(\left(\frac{2}{\lambda}-1  \right)e^{\frac{2x}{\lambda}}  \right)^{\left(1+\frac{\lambda}{\lambda-2}\right)} \nonumber \\
	&+2(1-\lambda)\left(\frac{2}{\lambda}-1\right)^{\frac{2(1-\lambda)}{\lambda-2}}e^{\frac{-2x}{\lambda-2}}    \W \left(\left(\frac{2}{\lambda}-1\right)e^{\frac{2x}{\lambda}} \right)^{1+\frac{\lambda}{\lambda-2}}, \nonumber
	\end{align}
	which is true, completing the result.
\end{proof}

\begin{theorem}\label{thm:entropyenergy}
	Let $f_0,f_1$ be defined as in Lemma~\ref{lem:leftandrightconjs} and let $\p$ be defined as in Lemma~\ref{lem:prox}. Then
	\begin{align}\label{proximalaverage}
	f_\lambda(x) &= \frac{\lambda-1}{2}\W \left(e^{\left(\frac{2\lambda-2}{2-\lambda}\W \left( \left(\frac{2}{\lambda}-1  \right)e^{\frac{2x}{\lambda}} \right)+\frac{2x}{\lambda} \right)} \right)^2 - \frac{x^2(\lambda-2)}{2\lambda}\\
	&+\resizebox{.85\textwidth}{!}{$%
		(\lambda-1) \W \left(e^{\left(\frac{2\lambda-2}{2-\lambda}\W \left( \left(\frac{2}{\lambda}-1  \right)e^{\frac{2x}{\lambda}} \right)+\frac{2x}{\lambda} \right)} \right) - \frac{\lambda(\lambda-1)^2}{(\lambda-2)^2}\W \left( \left(\frac{2}{\lambda}-1  \right)e^{\frac{2x}{\lambda}} \right)^2 . \nonumber $%
	}
	\end{align}
\end{theorem}

\begin{proof}
	Using Definition~\ref{def:pav} together with Lemma~\ref{lem:leftandrightconjs} we have that
	\begin{align*}
	f_\lambda = \left((1-\lambda)\left(\frac12 \W(e^{(\cdot)})(\W(e^{(\cdot)})+2)\right) + \lambda\left(\frac14(\cdot)^2 \right) \right)^* - \frac12 (\cdot)^2.
	\end{align*}
	This is just
	\begin{equation*}
	f_{\lambda} = \varphi - \frac{1}{2}(\cdot)^2
	\end{equation*}
	where $\varphi,\p_\varphi$ are as defined as in Lemma~\ref{lem:prox}. From this, we have that
	\begin{equation*}
	f_\lambda(x) = x\p_\varphi(x)- (1-\lambda)\left(\frac{1}{2}\W(e^{\p_\varphi(x)})(\W(e^{\p_\varphi(x)})+2) \right) -\frac{\lambda}{4}\p_\varphi(x)^2-\frac{1}{2}x^2,
	\end{equation*}
	which simplifies, by a great deal of arithmetic, to the form we see in \eqref{proximalaverage}, completing the result.
\end{proof}

Worthy of note is that this result (in particular, Lemma~\ref{lem:prox}) could not be computed by the \emph{SCAT} package. Neither could \emph{Maple} find the root of \eqref{functiontosolve} on its own. The solution was discovered by choosing specific values for $\lambda$, solving \eqref{functiontosolve}, observing, and finally guessing the more general pattern. This serves as an example of the kind of fruitful human-machine collaboration Borwein \& Lindstrom sought to emphasize in \cite{BL2016}.

Within minutes of choosing correctly we ``knew'' the answer, because we could visually read off the functions $f_0$ and $f_1$ at left in Figure~\ref{fig:entropy_energy}, even though a proof took much longer.

\begin{figure}
	\begin{center}
		\includegraphics[height=5cm]{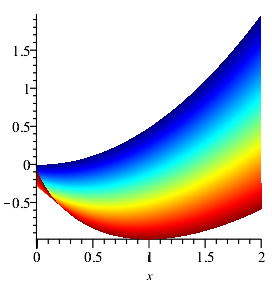}
		\includegraphics[height=5cm]{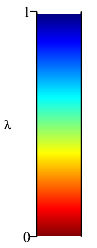}
		\includegraphics[height=5cm]{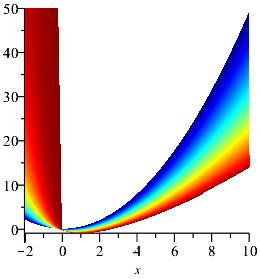}
	\end{center}

	\caption{$f_\lambda$ from Theorem~\ref{thm:entropyenergy}} \label{fig:entropy_energy}
\end{figure}

\subsection{Form and Proof for $f_\lambda^*$}\label{ss:flambdastar}

While the complicated nature of $f_\lambda$ precludes computing its conjugate in the usual way, we can still compute it using the convenient identity \eqref{conjugacy} found in  Remark~\ref{rem:conjugacy}. Specifically, since $f_\lambda^* = \mathcal{P}(f_0^*,\lambda,f_1^*)$, we can forget, for the moment, about $f_\lambda$ and instead compute $\mathcal{P}(f_0^*,\lambda,f_1^*)$ in the same way that we computed $f_\lambda$, directly from Definition~\ref{def:pav}.

\begin{remark}Let $f_0 = \ent$ and $f_1 = \frac{1}{2}|\cdot|^2$. Then
	\begin{equation*}
	f_0^* = \exp \quad \text{and} \quad f_1^* = \frac{1}{2}|\cdot|^2 = f_1.
	\end{equation*}
	These are both well-known results.
\end{remark}

\begin{lemma}\label{lem:leftandrightconjs*}Let $f_0 := \ent$ and $f_1 := \frac12 |\cdot|^2$. Then we have the following
	\begin{enumerate}[label=\emph{(\roman*)}]
		\item \label{leftconj*} $\left(f_0^* + \frac{1}{2}|\cdot|^2 \right)^* = \frac{1}{2}|\cdot|^2-\W(e^{(\cdot)})-\frac{1}{2}\W(e^{(\cdot)})^2$
		\item \label{rightconj*} $\left(f_1^* + \frac{1}{2}|\cdot|^2 \right)^* = \frac{1}{4}|\cdot|^2$
	\end{enumerate}
\end{lemma}
\begin{proof}
	\ref{leftconj*} By definition,
	\begin{equation}\label{innerterm}
	\left(f_0^* + \frac{1}{2}|\cdot|^2 \right)^*(x) = \underset{y \in {\small \R}}{\sup}\left\{xy-f_0^*(y) -\frac{1}{2}y^2 \right\}
	\end{equation}
	Differentiating the inner term with respect to $y$ and setting equal to zero, we have that the supremum is obtained when $y$ satisfies $e^y=x-y$. We will solve for $y$. Here the Wikipedia page about Lambert $\W$ suggests a handy method \cite{wikiW}. Let $\gamma = x-y$. Then $e^y = \gamma$ and so
	\begin{equation*}
	\gamma e^\gamma = e^y e^{x-y} = e^x
	\end{equation*}
	and so we have $\gamma = \W(e^x)$. Thus we have $e^y = \W(e^x)$.
	Taking the log of both sides,
	\begin{equation*}
	y = \log(\W(e^x)) = \log(e^x)-\W(e^x) = x-\W(e^x).
	\end{equation*}
	Using this as the $y$ value for the inner term in \eqref{innerterm}, we obtain
	\begin{equation*}
	\left(f_0^* + \frac{1}{2}|\cdot|^2 \right)^*(x) = x(x-\W(e^x))-\exp(x-\W(e^x)) -\frac{1}{2}\left(x-\W(e^x)\right)^2,
	\end{equation*}
	which simplifies to the form in \ref{leftconj*}.\newline

	\ref{rightconj*}: This is a well known result and may be obtained by simple arithmetic.
\end{proof}

\begin{lemma}\label{lem:prox*}
	Let $\theta$ be defined as follows
	\begin{align*}
	\theta &:= \left((1-\lambda)\left(\frac{1}{2}(\cdot)^2-\W(e^{(\cdot)})-\frac{1}{2}\W(e^{(\cdot)})^2 \right) + \lambda\left(\frac14(\cdot)^2 \right) \right)^*
	\end{align*}

	Then it holds that
	\begin{equation}\label{term0*}
	\p_\theta (x) = \left( \frac{2}{\lambda}-2\right) \W \left( \frac{\lambda e^{\frac{2x}{2-\lambda}}}{2-\lambda} \right)+\frac{2x}{2-\lambda}
	\end{equation}
	and so we may write
	\begin{equation*}
	\theta(x) = x\p_\theta(x)- (1-\lambda)\left(\frac{1}{2}\p_\theta(x)^2-\W(e^{\p_\theta(x)})-\frac{1}{2}\W(e^{\p_\theta(x)})^2 \right) -\frac{\lambda}{4}\p_\theta(x)^2.
	\end{equation*}
\end{lemma}

\begin{proof}
	Now by definition
	\begin{equation*}
	\theta(x) = \underset{y \in \R}{\sup}\left \{xy- (1-\lambda)\left(\frac{1}{2} y^2-\W(e^{y})-\frac{1}{2}\W(e^{y})^2 \right) -\frac{\lambda}{4} y^2 \right \},
	\end{equation*}
	which simplifies to
	\begin{equation*}
	\theta(x) = \underset{y \in \R}{\sup}\{xy+\frac{1}{2}(1-\lambda)\W(e^y)^2 +(1-\lambda)\W(e^y)+\frac{1}{4}(\lambda-2)y^2 \}.
	\end{equation*}
	Differentiating the inner term with respect to $y$ and setting equal to zero, we obtain
	\begin{equation}\label{exprmanual*}
	(1-\lambda)\W \left(e^y \right)+\left( \frac{1}{2}\lambda-1\right)y+x=0
	\end{equation}
	We will show that \eqref{exprmanual*} is true if $y=\p_\theta(x)$. First we will rewrite \eqref{exprmanual*} using the fact that $\W(a)=b$ if and only if $be^b = a$ which allows us to remove the $\W(e^y)$ term as follows:
	\begin{align*}
	\W(e^y)&=\frac{\left(1-\frac{\lambda}{2} \right)y-x}{1-\lambda}\\
	\text{and so}\quad e^y&=\left(\frac{\left(1-\frac{\lambda}{2} \right)y-x}{1-\lambda}\right) e^{\left(\frac{\left(1-\frac{\lambda}{2} \right)y-x}{1-\lambda} \right)}\\
	\text{which simplifies to}\quad 0 & = (\lambda y+2x-2y)e^{\left(\frac{(\lambda-2)y+2x}{2\lambda-2} \right)}-2e^y(\lambda-1).
	\end{align*}
	This is the form returned by \emph{Maple}. We further consolidate $y$ terms as follows,
	\begin{equation}\label{functiontosolve*}
	0=\big((\lambda-2)y+2x\big)\left(e^y\right)^{\left(\frac{\lambda-2}{2\lambda-2} \right)}e^{\left(\frac{x}{\lambda-1} \right)}-2(\lambda-1)e^y,
	\end{equation}
	which is the form we will use. We need only to show that
	\begin{equation}\label{WTS*}
	\left((\lambda-2)\p_\theta(x)+2x\right)e^{\left(\frac{x}{\lambda-1} \right)}\left(e^{\p_\theta(x)}\right)^{\left(\frac{\lambda-2}{2\lambda-2} \right)}-2(\lambda-1)e^{\p_\theta(x)}=0.
	\end{equation}
	First consider the term $e^{\p_\theta(x)}$. Since for any $a,b,z$ we have that
	$$e^{aW(z)+b}=(e^{W(z)})^a e^b = \left(\frac{z}{W(z)}\right)^a e^b $$
	we may let
	\begin{equation}
	a:=\left(\frac{2}{\lambda}-2\right), \quad b:=\frac{2x}{2-\lambda}, \quad z:=\frac{\lambda e^{\left(\frac{2x}{2-\lambda} \right)}}{2-\lambda}
	\end{equation}
	and thusly rewrite
	\begin{align}
	e^{\p_\theta(x)}&=\left(\frac{\lambda e^{\left(\frac{2x}{2-\lambda} \right)}}{(2-\lambda)\W \left(\frac{\lambda e^{\left(\frac{2x}{2-\lambda} \right)}}{2-\lambda}\right)} \right)^{\left(\frac{2}{\lambda}-2\right) }e^{\left(\frac{2x}{2-\lambda} \right)}\nonumber \\
	&= \left(2-\lambda \right)^{\left(2-\frac{2}{\lambda}\right)}\lambda^{\left(\frac{2}{\lambda}-2 \right)} e^{\left(\frac{2x}{\lambda}\right)} \W \left(\frac{\lambda e^{\left(\frac{2x}{2-\lambda} \right)}}{2-\lambda}\right)^{\left(2-\frac{2}{\lambda}\right)}.       \label{term1*}
	\end{align}
	From this, we have that
	\begin{align}
	e^{\left(\frac{x}{\lambda-1} \right)}\left(e^{\p(x)}\right)^{\left(\frac{\lambda-2}{2\lambda-2} \right)} &= e^{\left(\frac{x}{\lambda-1} \right)} \left(\left(\frac{\lambda e^{\left(\frac{2x}{2-\lambda} \right)}}{(2-\lambda)\W \left(\frac{\lambda e^{\left(\frac{2x}{2-\lambda} \right)}}{2-\lambda}\right)} \right)^{\left(\frac{2}{\lambda}-2\right) }e^{\left(\frac{2x}{2-\lambda} \right)} \right)^{\frac{\lambda-2}{2\lambda-2}} \nonumber \\
	&=e^{\left(\frac{x}{\lambda-1} \right)} \left(\frac{\lambda e^{\left(\frac{2x}{2-\lambda} \right)}}{(2-\lambda)\W \left(\frac{\lambda e^{\left(\frac{2x}{2-\lambda} \right)}}{2-\lambda}\right)} \right)^{\left(\frac{2-\lambda}{\lambda}\right) }e^{\left(-\frac{x}{\lambda-1} \right)} \nonumber \\
	&= \left(\frac{\lambda e^{\left(\frac{2x}{2-\lambda} \right)}}{(2-\lambda)\W \left(\frac{\lambda e^{\left(\frac{2x}{2-\lambda} \right)}}{2-\lambda}\right)} \right)^{\left(\frac{2-\lambda}{\lambda}\right) } \nonumber \\
	&=(2-\lambda)^{\left(\frac{\lambda-2}{\lambda}\right)}\lambda^{\left(\frac{2-\lambda}{\lambda} \right)}e^{\frac{2x}{\lambda}}\W \left(\frac{\lambda e^{\left(\frac{2x}{2-\lambda} \right)}}{2-\lambda} \right)^{\left(\frac{\lambda-2}{\lambda} \right)}.\label{term2*}
	\end{align}
	Using \eqref{term0*}, \eqref{term1*}, and \eqref{term2*}, we may rewrite \eqref{WTS*} as follows:
	\begin{align*}
	0&=(\lambda-2)\left(\left( \frac{2}{\lambda}-2\right) \W \left( \frac{\lambda e^{\frac{2x}{2-\lambda}}}{2-\lambda} \right)+\frac{2x}{2-\lambda}  \right)    (2-\lambda)^{\left(\frac{\lambda-2}{\lambda}\right)}\lambda^{\left(\frac{2-\lambda}{\lambda} \right)}e^{\frac{2x}{\lambda}}\W \left(\frac{\lambda e^{\left(\frac{2x}{2-\lambda} \right)}}{2-\lambda} \right)^{\left(\frac{\lambda-2}{\lambda} \right)}\\
	&+2x(2-\lambda)^{\left(\frac{\lambda-2}{\lambda}\right)}\lambda^{\left(\frac{2-\lambda}{\lambda} \right)}e^{\frac{2x}{\lambda}}\W \left(\frac{\lambda e^{\left(\frac{2x}{2-\lambda} \right)}}{2-\lambda} \right)^{\left(\frac{\lambda-2}{\lambda} \right)} \\
	&-2(\lambda-1)\left(2-\lambda \right)^{\left(2-\frac{2}{\lambda}\right)}\lambda^{\left(\frac{2}{\lambda}-2 \right)} e^{\left(\frac{2x}{\lambda}\right)} \W \left(\frac{\lambda e^{\left(\frac{2x}{2-\lambda} \right)}}{2-\lambda}\right)^{\left(2-\frac{2}{\lambda}\right)}.\\
	\end{align*}
	The positive and negative terms of the form
	\begin{equation*}
	2x(2-\lambda)^{\left(\frac{\lambda-2}{\lambda}\right)}\lambda^{\left(\frac{2-\lambda}{\lambda} \right)}e^{\frac{2x}{\lambda}}\W \left(\frac{\lambda e^{\left(\frac{2x}{2-\lambda} \right)}}{2-\lambda} \right)^{\left(\frac{\lambda-2}{\lambda} \right)}
	\end{equation*}
	cancel each other out, leaving
	\begin{align*}
	0&=(\lambda-2)\left(\left( \frac{2}{\lambda}-2\right) \W \left( \frac{\lambda e^{\frac{2x}{2-\lambda}}}{2-\lambda} \right)  \right)    (2-\lambda)^{\left(\frac{\lambda-2}{\lambda}\right)}\lambda^{\left(\frac{2-\lambda}{\lambda} \right)}e^{\frac{2x}{\lambda}}\W \left(\frac{\lambda e^{\left(\frac{2x}{2-\lambda} \right)}}{2-\lambda} \right)^{\left(\frac{\lambda-2}{\lambda} \right)}\\
	&-2(\lambda-1)\left(2-\lambda \right)^{\left(2-\frac{2}{\lambda}\right)}\lambda^{\left(\frac{2}{\lambda}-2 \right)} e^{\left(\frac{2x}{\lambda}\right)} \W \left(\frac{\lambda e^{\left(\frac{2x}{2-\lambda} \right)}}{2-\lambda}\right)^{\left(2-\frac{2}{\lambda}\right)},
	\end{align*}
	which further simplifies to
	\begin{align*}
	0&=2(\lambda-2)\left( \frac{1}{\lambda}-1\right)    \left(\frac{2-\lambda}{\lambda}\right)^{\left(\frac{\lambda-2}{\lambda}\right)}e^{\left(\frac{2x}{\lambda}\right)}\W \left(\frac{\lambda e^{\left(\frac{2x}{2-\lambda} \right)}}{2-\lambda} \right)^{\left(2-\frac{2}{\lambda} \right)}\\
	&-2(\lambda-1)\left(\frac{2-\lambda}{\lambda} \right)^{\left(2-\frac{2}{\lambda}\right)} e^{\left(\frac{2x}{\lambda}\right)} \W \left(\frac{\lambda e^{\left(\frac{2x}{2-\lambda} \right)}}{2-\lambda}\right)^{\left(2-\frac{2}{\lambda}\right)}.
	\end{align*}
	Finally, $(\lambda-2)\left(\frac{1}{\lambda}-1 \right) = (\lambda-1)\left(\frac{2-\lambda}{\lambda} \right)$ and so the above equation is true, completing the result.
\end{proof}

\begin{theorem}\label{thm:expenergy}
	Let $f_0,f_1$ be defined as in Lemma~\ref{lem:leftandrightconjs*}. Then

	\begin{align}\label{proximalaverageconjugate}
	f_\lambda^*(x) &= (1-\lambda)\W \left(e^{\left( \left(\frac{2}{\lambda}-2 \right)\W \left( \frac{\lambda}{2-\lambda} e^{\left(\frac{2x}{2-\lambda} \right)} \right)+\frac{2x}{2-\lambda}    \right)}\right) + \frac{\lambda x^2}{4-2\lambda}\\
	&+ \frac{1}{2}(1-\lambda)\W \left(e^{\left( \left(\frac{2}{\lambda}-2 \right)\W \left( \frac{\lambda}{2-\lambda} e^{\left(\frac{2x}{2-\lambda} \right)} \right)+\frac{2x}{2-\lambda}    \right)}\right)^2\\
	&+ \frac{(\lambda-1)^2(\lambda-2)}{\lambda^2}\W \left( \frac{\lambda}{2-\lambda} e^{\left(\frac{2x}{2-\lambda} \right)} \right)^2. \nonumber
	\end{align}
\end{theorem}
\begin{proof}
	Using Definition~\ref{def:pav} together with Lemma~\ref{lem:leftandrightconjs*} we have that
	\begin{align*}
	f_\lambda^* = \left((1-\lambda)\left( \frac{1}{2}|\cdot|^2-\W(e^{(\cdot)})-\frac{1}{2}\W(e^{(\cdot)})^2 \right) + \lambda\left(\frac14(\cdot)^2 \right) \right)^* - \frac12 (\cdot)^2.
	\end{align*}
	This is just
	\begin{equation*}
	f_{\lambda}^* = \theta - \frac{1}{2}(\cdot)^2
	\end{equation*}
	where $\theta,\p_\theta$ are as in Lemma~\ref{lem:prox*}. From this, we obtain
	\begin{align*}
	f_\lambda^*(x) &= x\p_\theta(x)- (1-\lambda)\left(\frac{1}{2}\p_\theta(x)^2-\W(e^{\p_\theta(x)})-\frac{1}{2}\W(e^{\p_\theta(x)})^2 \right) \\
	&-\frac{\lambda}{4}\p_\theta(x)^2-\frac{1}{2}x^2.\nonumber \end{align*}
	This simplifies, by a great deal of arithmetic, to the form we see in \eqref{proximalaverageconjugate}, completing the result.
\end{proof}

Similarly to Theorem~\ref{thm:entropyenergy}, the results admitting Theorem~\ref{thm:expenergy} (in particular, Lemma~\ref{lem:prox*}) could not be obtained through the use of \emph{SCAT} or \emph{Maple} alone because these packages cannot invert \eqref{functiontosolve*}. The solution was again discovered with a method similar to that of Theorem~\ref{thm:entropyenergy}.

Again within minutes of choosing correctly we ``knew'' the answer, because we could visually read off the functions $f_0^*$ and $f_1^*$ in Figure~\ref{fig:exp_energy}, even though a proof took much longer. Figures~\ref{fig:entropy_energy} and \ref{fig:exp_energy} highlight an advantageous characteristic of the proximal average which we provide in the following remark.

\begin{remark}\label{rem:fulldomain}
	Let $f_0,f_1 \in \mathcal{F}$ and $\lambda \in \left]0,1\right[$. Let $f_\lambda := \mathcal{P}\left(f_0,\lambda,f_1 \right).$ Suppose that $f_0$ or $f_1$ has full domain and that $f_0^*$ or $f_1^*$ has full domain. Then the following hold:
	\begin{enumerate}
		\item Both $f_\lambda$ and $f_\lambda^*$ have full domain.
		\item If $f_0$ or $f_1$ is differentiable everywhere, then so is $f_\lambda$.
		\item If $f_0$ or $f_1$ is strictly convex and its Fenchel conjugate has full domain, then $f_\lambda$ is strictly convex.
	\end{enumerate}
	For a proof, see \cite[Theorem 6.2]{BLT2007}.
\end{remark}

Figures~\ref{fig:entropy_energy} and \ref{fig:exp_energy} also illustrate another important difference between the behaviour of limiting cases for the proximal average and for the ordinary average.

\begin{figure}
	\begin{center}
		\includegraphics[height=5cm]{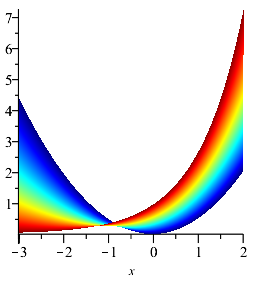}
		\includegraphics[height=5cm]{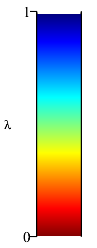}
		\includegraphics[height=5cm]{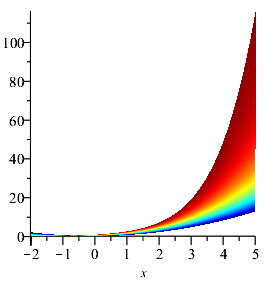}
	\end{center}

	\caption{$f_\lambda^*$ from Theorem~\ref{thm:expenergy}} \label{fig:exp_energy}
\end{figure}

\begin{remark}[Limiting Cases for Proximal Average]\label{rem:limitingcasesproximal}In juxtaposition with Remark~\ref{originallimitingcases}, we obtain different limiting cases for $f_\lambda$ \eqref{proximalaverage} and $f_\lambda^*$ \eqref{proximalaverageconjugate}. For $f_\lambda$, in the limits at $0,1$, we reobtain $f_0$ and $f_1$ respectively. For $f_\lambda^*$ in the limits at $0,1$, we reobtain $\exp$ and $\frac12|\cdot|^2$. This is more natural, because these are $f_0^*$ and $f_1^*$ respectively, and so our continuous transformation of our functions has corresponded with a continuous transformation of their conjugates.
\end{remark}

The juxtaposition in Remark~\ref{rem:limitingcasesproximal} is both an immediate consequence of and also an excellent illustration of Remark~\ref{rem:continuity}. Where $f_0,f_1 \in \mathcal{F}$ and $(\lambda_n)_{n\in \mathbb{N}}$ is a sequence in $[0,1]$, we have from Remark~\ref{rem:continuity} that
\begin{align*}
\text{if} \;\; \lambda_n \rightarrow 0 \quad \text{then}\quad \mathcal{P}(f_0,\lambda_n,f_1) \overset{\scriptsize \e}{\rightarrow}& \mathcal{P}(f_0,0,f_1) = f_0 \\
\text{and}\quad \mathcal{P}(f_0^*,\lambda_n,f_1^*) \overset{\scriptsize \e}{\rightarrow}& \mathcal{P}(f_0^*,0,f_1^*) = f_0^*\\
\text{and if} \;\; \lambda_n \rightarrow 1 \quad \text{then} \quad \mathcal{P}(f_0,\lambda_n,f_1) \overset{\scriptsize \e}{\rightarrow}& \mathcal{P}(f_0,1,f_1) = f_1\\
\text{and} \quad \mathcal{P}(f_0^*,\lambda_n,f_1^*) \overset{\scriptsize \e}{\rightarrow}& \mathcal{P}(f_0^*,1,f_1^*) = f_1^*,
\end{align*}
which is both elegant and convenient.

\section{Minimizing an Entropy Functional}\label{s:entropyminimization}

In their 2016 paper \cite{BL2016} Borwein \& Lindstrom illustrated the utility of the Lambert $\W$ function by showing how it naturally arises in the problem of minimizing an entropy functional of the form
\begin{align*}
I_f : L^1([0,1]) &\rightarrow \mathbb{R} \\
\text{by} \quad I_f : x &\mapsto \int_{0}^{1} f(x(s)) \d
\end{align*}
where $f$ is a proper, closed convex function. The problem is to minimize $I_f$ subject to finitely many continuous linear constraints of the form
\begin{equation*}
\langle a_k,x \rangle = \int_0^1 a_k(s)x(s) \d = b_k
\end{equation*}
for $1 \leq k \leq n$. We may write this linear equality constraint concisely as
\begin{align*}
A:L^1([0,1]) &\rightarrow \mathbb{R}^n \\
\text{by}\quad A: x & \mapsto \left(\int_0^1 a_1(s)x(s) \d , \dots , \int_0^1 a_n(s)x(s) \right)=b\\
\text{where}\quad b:= A\rho
\end{align*}
where $\rho, a_k \in L^\infty ([0,1])$ and $\rho$ is a given function used to generate the data vector $b$. When $f^*$ is smooth and everywhere finite on the real line, the problem
\begin{equation}\label{primalproblem}
\underset{x\in L^1}{\inf}\{I_f(x) | Ax=b \}
\end{equation}
reduces to solving a finite nonlinear equation
\begin{equation}\label{optimalmultipliers}
\int_0^1 (f^*)' \left(\sum_{j=1}^n \mu_j a_j(s) \right) a_k(s) \d = b_k \quad (1\leq k \leq n).
\end{equation}
A discussion of why this is the case is given in \cite[Section 7]{BL2016} which employs results from Jonathan Borwein's works co-authored with Adrian Lewis \cite{BL2000}, Qiji Zhu \cite{BZ2005}, and Jon Vanderwerff \cite{BV2010}, and Liangjin Yao \cite{BY2014}. The matter of primal attainment and constraint qualification are addressed in Borwein's and Lewis' article \cite{BL1991}, as well as in Lindstrom's PhD dissertation \cite{ScottPhD}.

As was also true in the setting of \cite{BL2016}, this problem and methods discussed in this section are informed by methods found in all of these works, to which we refer the reader for additional information about any underlying theory.

For the function $f$ in the construction of $I_f$, Borwein et al. opted to use $f_t$ from \eqref{originalaverage}, for which the corresponding $f_t^*$ has the form in \eqref{originalconjugate} for $0<t<1$ and $f_t^* = \exp,\; \frac{1}{2}|\cdot|^2$ for $t=0,1$, respectively. For this choice:

\begin{align*}
\text{for }\; 0<t<1, \quad (f_t^*)'(x)&= \frac{1-t}{t}\W \left(\frac{t}{1-t}\exp\left(\frac{x}{1-t}\right) \right)\\
\text{for }\; t=0, \quad (f_t^*)'(x)&=\exp(x)\\
\text{for }\; t=1, \quad (f_t^*)'(x)&= x.
\end{align*}
In the limiting case as $t$ approaches $0$, $(f_t^*)$ approaches $\exp$ while in the limiting case as $t$ approaches $1$ we obtain $\max \{0,x\}$, given the discussion of the limiting cases of $f_t^*$ in Remark~\ref{originallimitingcases}.

\begin{remark}\label{rem:partialdomain}
	Let $f:X \rightarrow \left]-\infty,+\infty \right[$ be proper. Then $f^*:X \rightarrow \left]-\infty,+\infty \right]$ is proper. Let $x,u \in X$. Then
	$$
	u \in \partial f^*(x) \iff f(u)+f^*(x)= \langle x,u \rangle \iff x \in \partial f(u).
	$$
	For details, see \cite[proposition 16.9]{BC}. Thus we have that
	$$
	{\rm ran} (\partial f^*) \subset {\rm dom}(\partial f) \subset {\rm dom}(f).
	$$
	Consequently, for all $x\in H$ we have that:
	$$
	\left(\;\forall t \in \left[0,1\right[\; \right) \quad f_t^*(x) \in {\rm dom}(f_t) = \left [0,\infty \right[.
	$$
\end{remark}

For the function $f$ in the construction of $I_f$, we consider $f_\lambda$ from \eqref{proximalaverage}, for which the corresponding $f_\lambda^*$ has the form in \eqref{proximalaverageconjugate} for $0<\lambda<1$ and $f_\lambda^* = \exp ,\; \frac{1}{2}|\cdot|^2$ for $\lambda=0,1$ respectively as explained in  Remark~\ref{rem:limitingcasesproximal} and as follows from Theorem~\ref{thm:expenergy} by differentiation. For this choice:

\begin{align*}
(f_0^*)'(x)&=\exp(x)\\
(f_1^*)'(x)&= x\\
\text{and for}&\;\; 0<\lambda<1,\\
(f_\lambda^*)'(x)&=\frac{1}{1+\omega(x)}\Bigg(\frac{2(1-\lambda)}{\lambda}\left(\omega(x)-\frac{\lambda}{\lambda-2} \right) \W \left(e^{\left(\frac{2}{\lambda}-2 \right)\omega(x)+\frac{2x}{2-\lambda}} \right)\\
&-\frac{4(\lambda-1)^2}{\lambda^2}\omega(x)^2 + \frac{\lambda x}{2-\lambda}\omega(x) +\frac{x\lambda}{2-\lambda} \Bigg)\\
\text{where }\quad \omega(x)&=\W \left(\frac{\lambda}{2-\lambda}e^{\frac{2x}{2-\lambda}} \right).
\end{align*}

\begin{figure}
	\begin{center}
		\includegraphics[height=5cm]{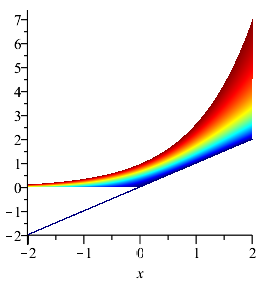}
		\includegraphics[height=5cm]{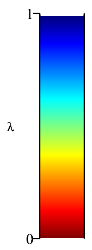}
		\includegraphics[height=5cm]{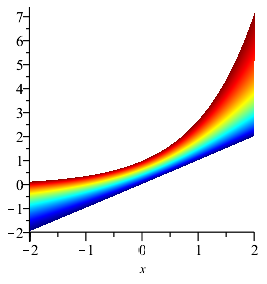}
	\end{center}

	\caption{$(f_t^*)'$ (left) and $(f_\lambda^*)'$ (right)} \label{fig:fstarprime}
\end{figure}

The functions $(f_t^*)'$ and $(f_\lambda^*)'$ may be seen in Figure~\ref{fig:fstarprime}. Figure~\ref{fig:fstarprime} also serves to highlight one of the consequences of Remark~\ref{rem:fulldomain} in our case.

In juxtaposition with $f_t$ which takes the value infinity for all negative real values, $f_\lambda$ has full domain because $f_1$ has full domain. Consequently the conjugate $f_\lambda^*$ of $f_\lambda$ decreases on part of its domain for values of $\lambda \in ]0,1]$; this is in contrast with the conjugate $f_t^*$ of $f_t$ which is nondecreasing except for the case $t=1$. As a result, the image of $(f_\lambda^*)'$ contains negative numbers for $\lambda \in \left]0,1\right]$ while the image of $(f_t^*)'$ contains negative numbers only for $t=1$. In terms of Remark~\ref{rem:partialdomain}, $f_\lambda$ differs from $f_t$ in the sense that
$$
\left(\forall x \in \mathcal{H} \right)\quad \left(\forall \lambda \in \left]0,1 \right] \right)\quad f_\lambda^*(x) \in {\rm dom}(f_\lambda) = \left]-\infty,\infty \right[.
$$

\begin{remark}\label{rem:error}
	In their original article \cite{BL2016}, the authors have labelled solutions computed for the \emph{limiting} case $\underset{t\rightarrow 1}{\lim}(f_t^*)' = \max \{\cdot,0 \}$ with the label $t=1$; however, $(f_0^*)'$ is actually just the identity $x \mapsto x$. This labelling confusion does not change any of the key results of the paper; it affects only computed examples.
\end{remark}

Where $\mu_1,\dots,\mu_n$ are the optimal multipliers in \eqref{optimalmultipliers}, the primal solution $x_\lambda$ to the primal problem \eqref{primalproblem} is then given by
\begin{equation*}
x_\lambda(s) = \left(f_\lambda^*\right)' \left(\sum_{j=1}^n \mu_j a_j(s) \right)
\end{equation*}

A key difference between our setting and that of \cite{BL2016} is then immediately apparent: for $t\ne 1$, the primal solutions when optimizing with the conventional average $f_t$ could not take on negative values. When instead using the proximal average, $f_\lambda$, the primal solutions may take on negative values so long as $\lambda \ne 0$. The hard barrier (or lack of hard barrier) against negative values may be considered either an advantage or disadvantage depending upon one's intentions.

\subsection{Computed Examples}\label{ss:entropyexamples}

For all examples where we solve \eqref{primalproblem}, we compute with $8$ moments ($n=8$), and we follow the lead of Borwein \& Lindstrom \cite{BL2016}, employing a Gaussian quadrature with 20 abscissas for the numerical integration necessary to solve the system~\eqref{optimalmultipliers}. One may consult Borwein \& Lindstrom \cite{BL2016} for an index on computation which explains a simple implementation with Newton's method. When reporting solutions for the weighted average $f_t$, instead of the case where $t=1$, we choose to plot the limiting case:
\begin{equation*}
\underset{t\rightarrow 1}{\lim}(f_t^*)' = \max \{\cdot, 0 \}.
\end{equation*}
The first reason for this is that the exact cases $t=1$ and $\lambda =1$ coincide (see Remark~\ref{rem:limitingcasesproximal}), and so comparing them is not as interesting. The second reason is to be consistent with the method of reporting employed in \cite{BL2016} (see Remark~\ref{rem:error}).

We compute with vertical translations of the function we wish to reconstruct, the function used by Borwein \& Lindstrom,
\begin{equation}\label{objectivefunction}
\rho: s \mapsto \frac{3}{5} + \frac{1}{2}\sin \left(3\pi s^2 \right),
\end{equation}
which we compute with in Example~\ref{ex:1}.

\begin{figure}
	\begin{center}
		\includegraphics[width=.4\textwidth]{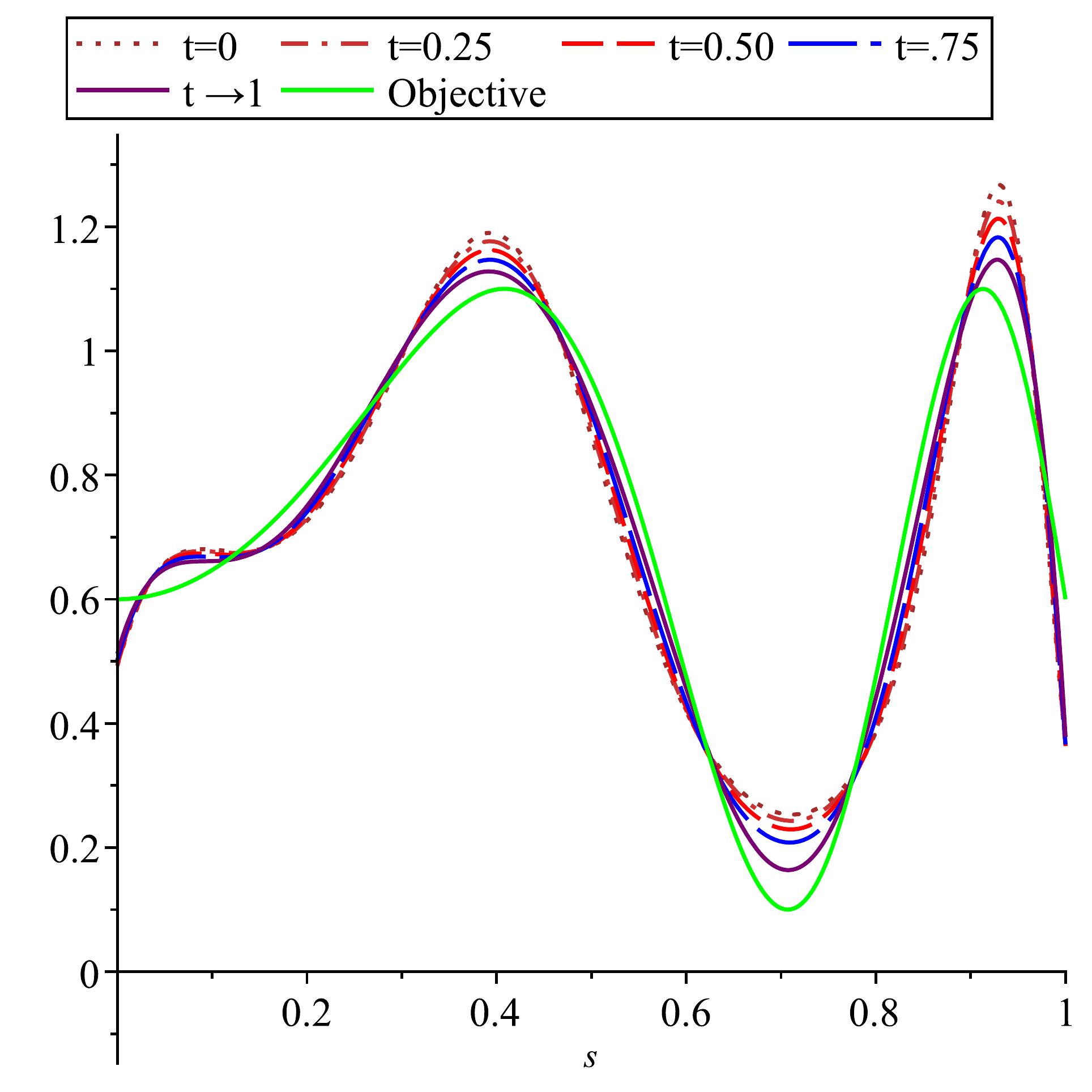}\quad \includegraphics[width=.4\textwidth]{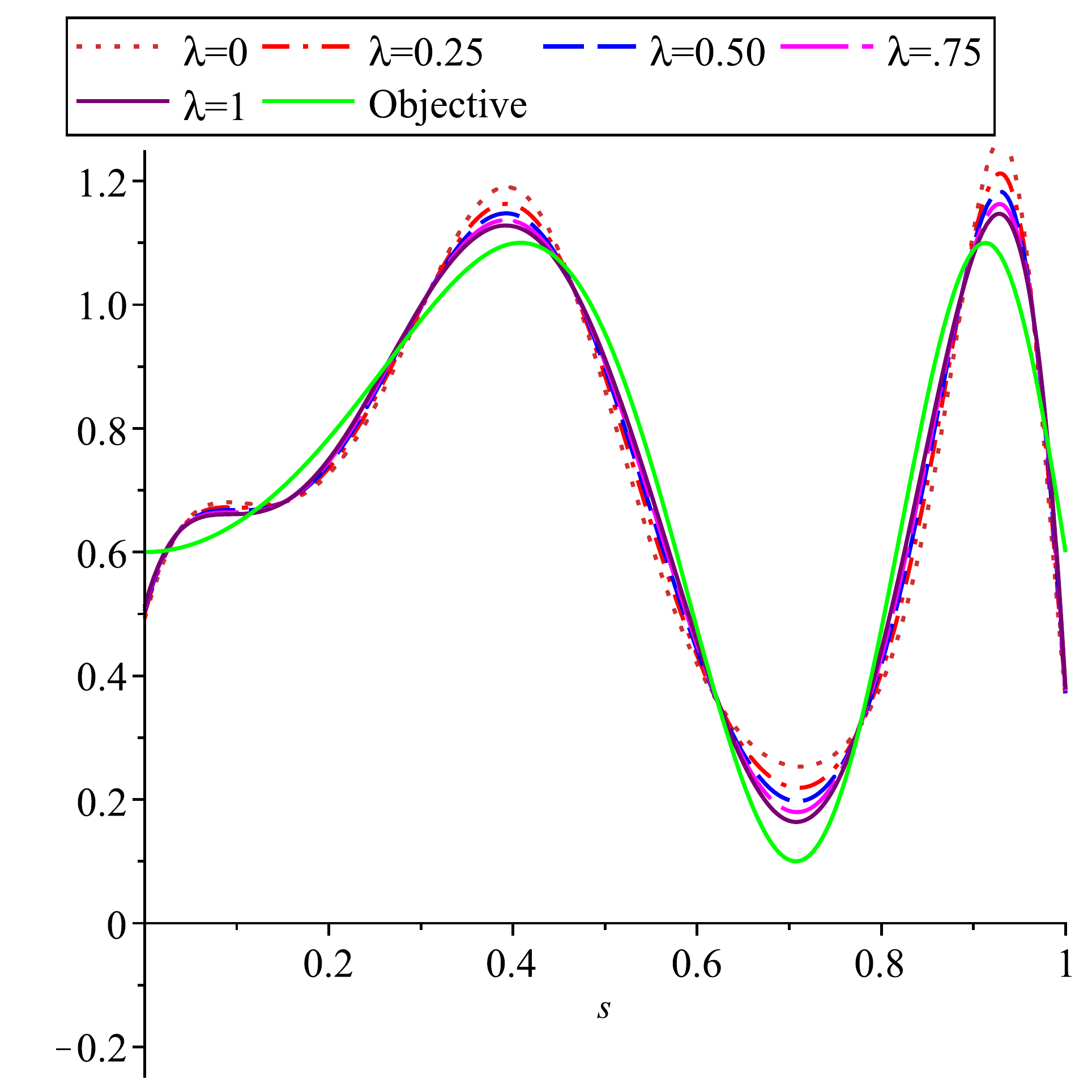}
	\end{center}
	\caption{Primal solutions from Example~\ref{ex:1} appear quite similar.}\label{fig:ex1}
\end{figure}

\begin{example}[Similarities between weighted average and proximal average]\label{ex:1}
Figure~\ref{fig:ex1} shows similar-looking primal solutions obtained by computing with the weighted average $f_t$ and the proximal average $f_\lambda$ where the objective function is as in \eqref{objectivefunction}. Importantly, in this case $\rho(s) \geq 0 \; \forall s \in [0,1]$.
\end{example}

The advantage of homotopy becomes apparent when the objective function has negative output values, as it does in the next example.

\begin{figure}
	\begin{center}
		\includegraphics[width=.4\textwidth]{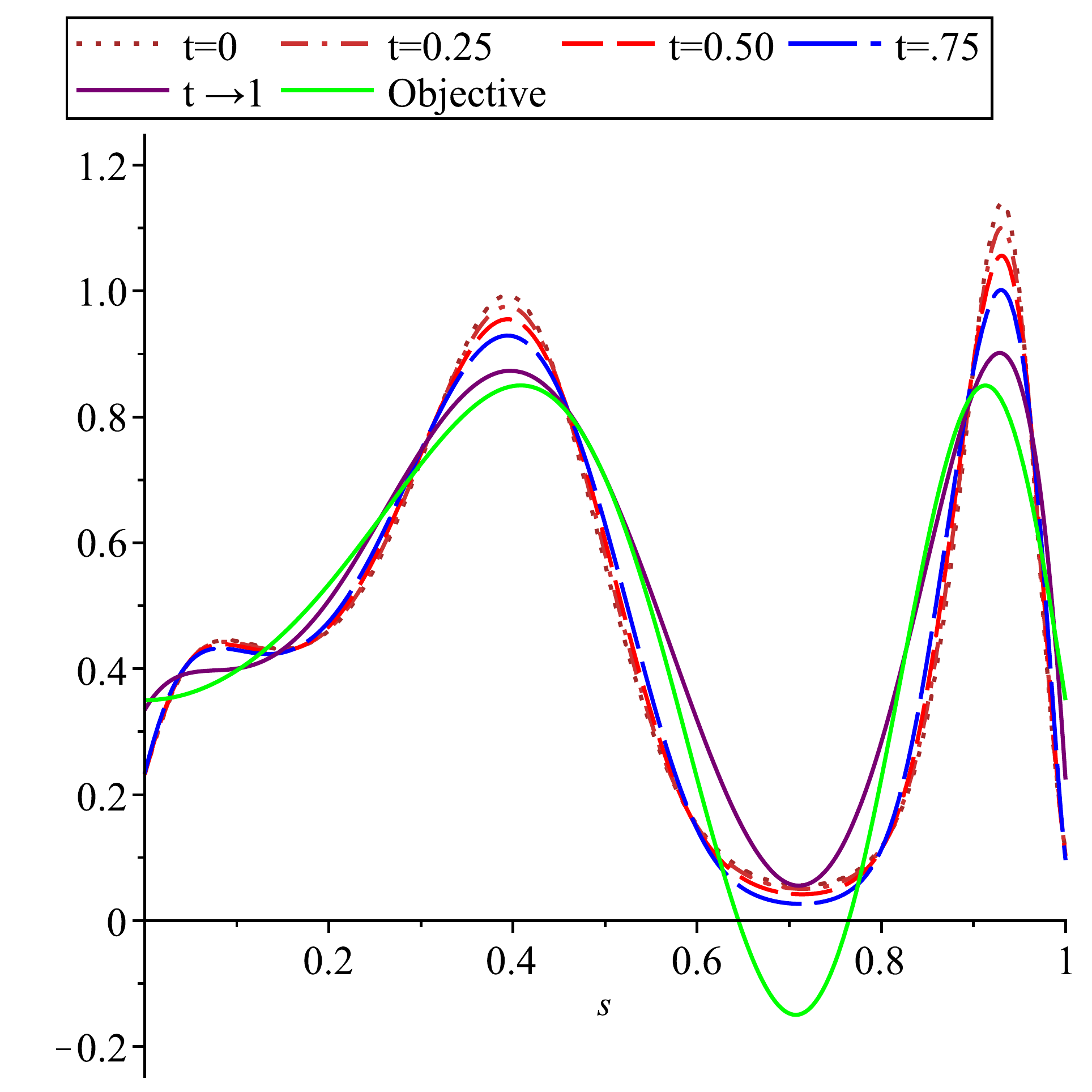}\quad \includegraphics[width=.4\textwidth]{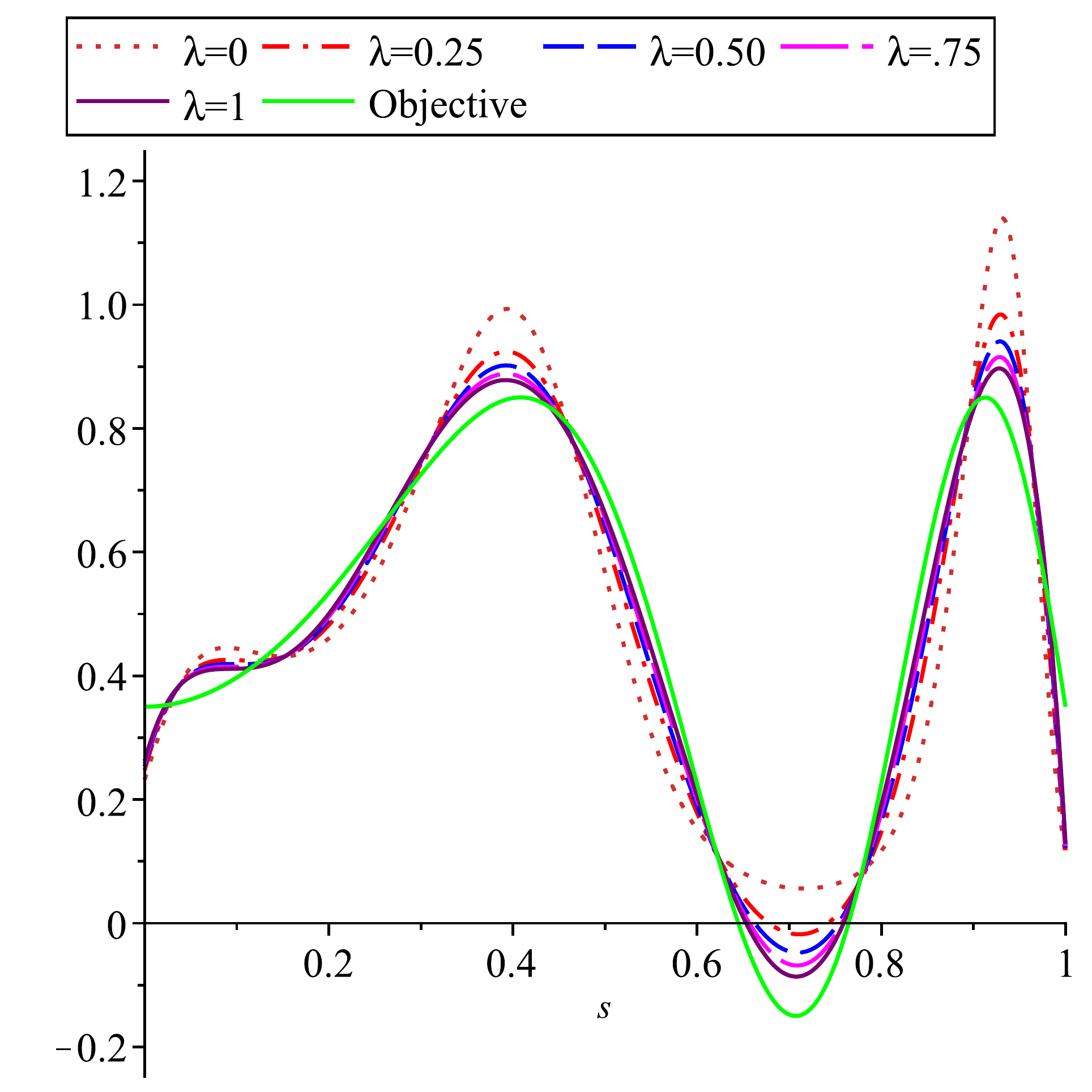}
	\end{center}
	\caption{Primal solutions from Example~\ref{ex:2} are noticeably different, particularly where the objective function is negative-valued.}\label{fig:ex2}
\end{figure}

\begin{example}[Differences between weighted average and proximal average]\label{ex:2}
	For the second example, we compute with a negative translation of the previous objective function:
	\begin{equation*}
	\rho: s \mapsto \frac{7}{20} + \frac{1}{2}\sin \left(3\pi s^2 \right).
	\end{equation*}
	Figure~\ref{fig:ex2} shows the primal solutions for the weighted average $f_t$ at left and for the proximal average $f_\lambda$ at right.

	The presence of negative values for the objective function $\rho$ illuminates an important advantage of the proximal average $f_\lambda$. Because $(f_\lambda^*)'$ is allowed to have negative range values for $\lambda > 0$ (as shown in Figure~\ref{fig:fstarprime}), the primal solutions in the proximal average case are able to have negative range values for $\lambda >0$. As a result, the primal solutions corresponding to the proximal average with $\lambda > 0$ are a better fit for our objective function $\rho$ than the primal solutions corresponding to the weighted average.
\end{example}

For the advantage of homotopy---that primal solutions may take on negative values when $\lambda \ne 0$---there is a price to pay computationally. Namely, in contradistinction with the case of the weighted average $f_t$, Newton's method no longer reliably solves the problem for the proximal average $f_\lambda$ when the objective function is permitted to take values below or near zero. This is shown in Example~\ref{ex:3}.

\begin{figure}
	\begin{center}
		\includegraphics[width=.4\textwidth]{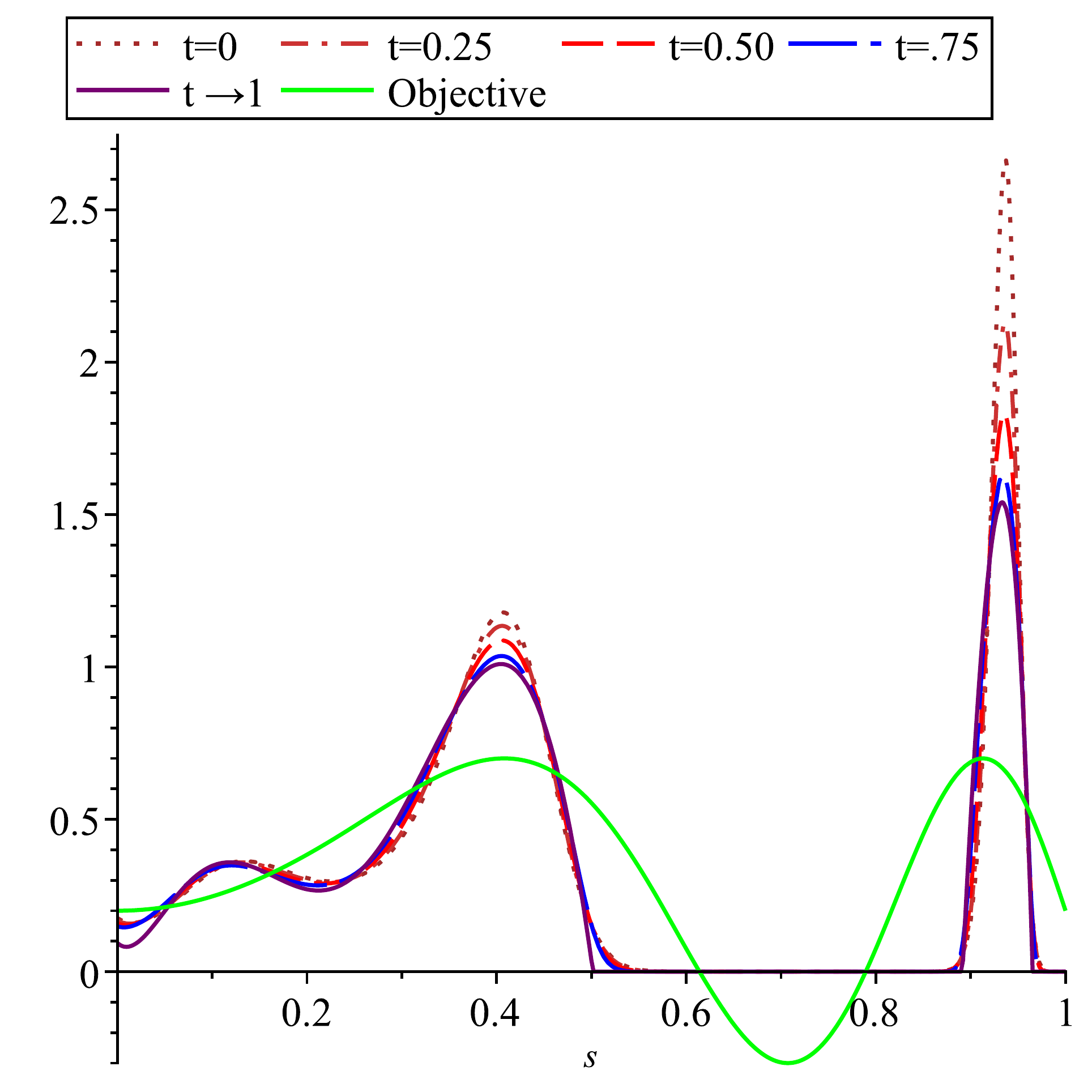}\quad \includegraphics[width=.4\textwidth]{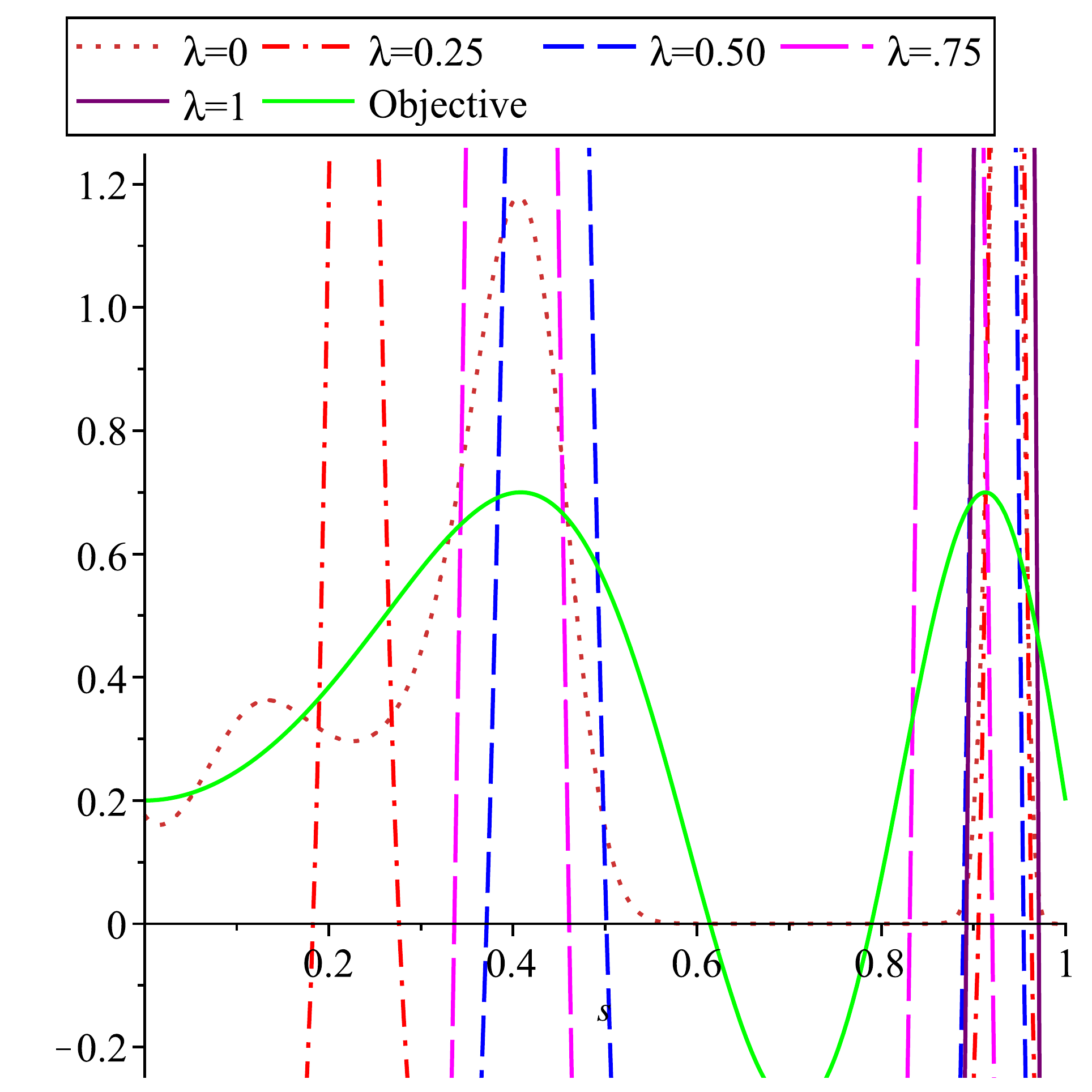}
	\end{center}
	\caption{Newton's method is less reliable for the proximal average in the case of Example~\ref{ex:3}.}\label{fig:ex3}
\end{figure}

\begin{example}[Computational challenge]\label{ex:3}
	To illustrate a computational disadvantage of homotopy, we compute with the data vector $b$ generated by
	\begin{equation*}
	\rho: s \mapsto \frac{1}{5} + \frac{1}{2}\sin \left(3\pi s^2 \right),
	\end{equation*}
	which is another downward translation of the function used to generate the data vector in Example~\ref{ex:2}. With the starting point of $(\frac{1}{2},\dots \frac{1}{2}) \in \mathbb{R}^8$, Newton's method fails to find the optimal solution for the homotopy $f_\lambda$ while it still manages to find the optimal solution for the weighted average $f_t$. Figure~\ref{fig:ex3} shows the primal solution for $f_t$ at left and the primal output when Newton's method is paused after 400 iterates for $f_\lambda$ at right.
\end{example}

Rather than using Newton's method, one might instead use gradient descent to solve the system~\eqref{optimalmultipliers}, either by seeking to
\begin{enumerate}[label=(\roman*)]
	\item \label{maximizethedual} solve the dual problem directly
	\item \label{sumofsquaresgradients} minimize the sum of the squares of the gradient components.
\end{enumerate}
\ref{maximizethedual}: In the former case, the gradient we use has the $n$ components
\begin{equation*}
\int_0^1 (f^*)' \left(\sum_{j=1}^n \mu_j a_j(s) \right) a_k(s) \d - b_k \quad (1\leq k \leq n),
\end{equation*}
which is, of course, the system from \ref{optimalmultipliers}. \\
\ref{sumofsquaresgradients}: In the latter case, the problem becomes:
\begin{align*}
\text{Find}\; &\mu \in \mathbb{R}^n \;\text{such that}\; \mathcal{G}(\mu):=\sum_{k=1}^{n} \mathcal{G}_k(\mu) = 0 \; \text{where} \\
 \mathcal{G}_k&: \mu \mapsto \left(\int_0^1 (f^*)' \left(\sum_{j=1}^n \mu_j a_j(s) \right) a_k(s) \d -b_k \right)^2, (1\leq k \leq n).
\end{align*}
Again using a Gaussian quadrature rule with $m$ abscissas $s_1,\dots,s_m$ and corresponding weights $w_1,\dots,w_m$, we let
\begin{equation*}
\sum_{i=1}^{m}w_i (f^*)' \left(\sum_{j=1}^n \mu_j a_j(s_i) \right) a_k(s_i) :\approx \int_0^1 (f^*)' \left(\sum_{j=1}^n \mu_j a_j(s) \right) a_k(s) \d
\end{equation*}
and so \eqref{optimalmultipliers} reduces to finding $\mu \in \mathbb{R}^n$ such that:
\begin{align}\label{gradientdescentproblem}
\mathcal{G}(\mu)=\sum_{k=1}^n \left(\left(\sum_{i=1}^{m}w_i (f^*)' \left(\sum_{j=1}^n \mu_j a_j(s_i) \right) a_k(s_i) \right) -b_k \right)^2 =0.
\end{align}
To solve \eqref{gradientdescentproblem}, we may use gradient descent where
\begin{equation*}
\nabla \mathcal{G}(\mu) = \left(\frac{\partial}{\partial \mu_1}\mathcal{G}(\mu),\dots,\frac{\partial}{\partial \mu_n}\mathcal{G}(\mu)\right).
\end{equation*}

\begin{example}[Gradient Descent]\label{ex:4}
	When we implement gradient descent for either of the above approaches with the same starting point and objective function from Example~\ref{ex:3}, the method tends to stall. Consequently, the primal solutions yielded do not correspond to the true solution for the problem and only roughly resemble the function $\rho$ used to generate the data. This is shown at right in Figure~\ref{fig:ex4}.
\end{example}

\begin{figure}
	\begin{center}
		\includegraphics[width=.4\textwidth]{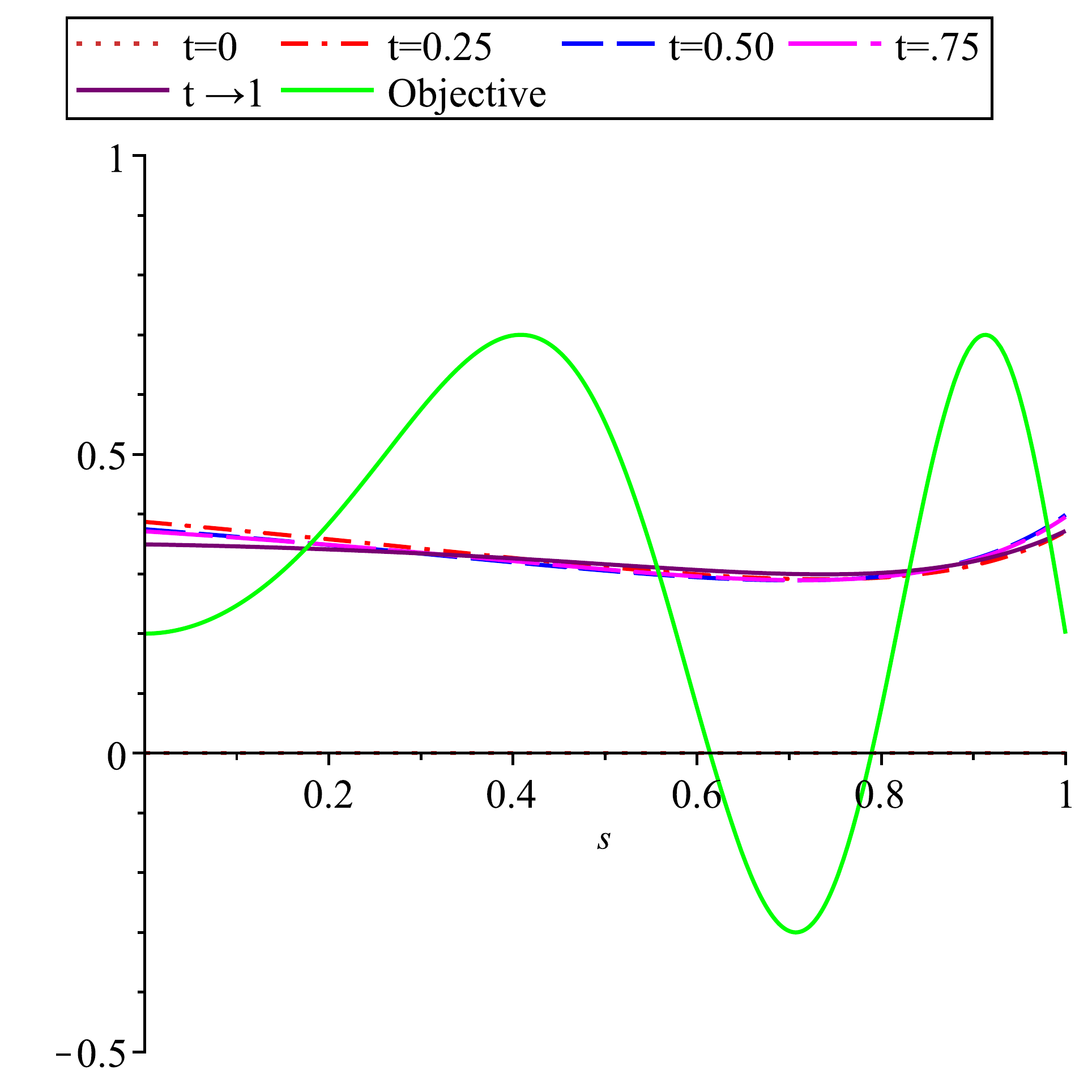}\quad \includegraphics[width=.4\textwidth]{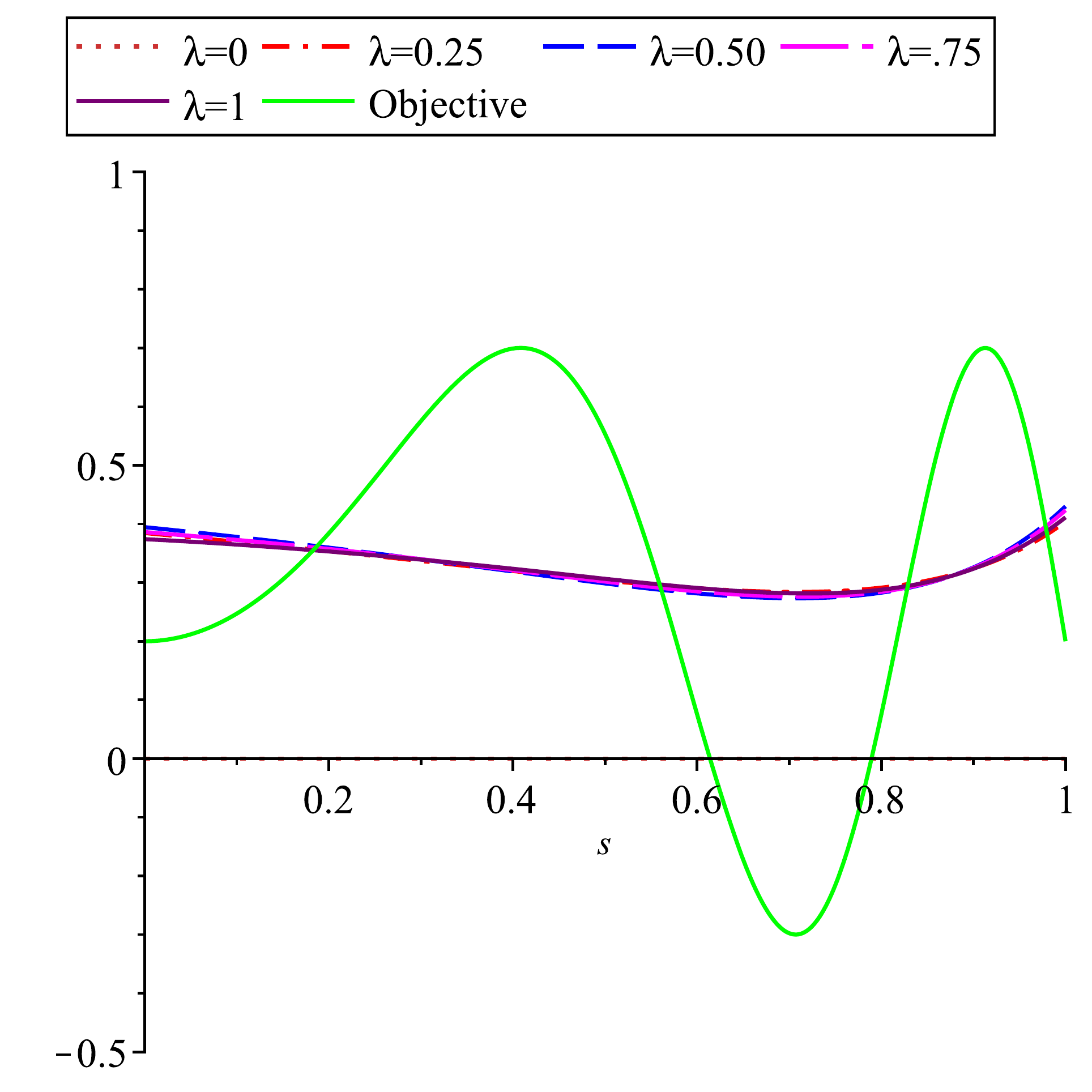}
	\end{center}
	\caption{Gradient descent may be insufficient as in Example~\ref{ex:4}.}\label{fig:ex4}
\end{figure}

\subsection{A homotopy method}

As a remedy for stalling, we may employ a homotopy-type whereby we solve a \emph{sequence of problems}. Suppose we seek a solution where the objective function is given by
\begin{equation*}
\rho : s \mapsto \frac{7}{20} + \frac{1}{2}\sin \left(3\pi s^2\right)- \Delta.
\end{equation*}
Then we let
\begin{equation*}
\rho_N : s \mapsto \frac{7}{20} + \frac{1}{2}\sin \left(3\pi s^2\right)-N \delta, \;\; N \in \{0,\dots,\upsilon\}\subset \mathbb{N}, \;\; \delta >0 , \;\; \upsilon \delta = \Delta.
\end{equation*}
We further define $\mu^N$ be the solution to \eqref{optimalmultipliers} for the problem corresponding to the linear constraint generated by the function $\rho_N$. We can find $\mu^0$ with Newton's method (and did so in Example~\ref{ex:2}). We may then using $\mu^0$ as our \emph{starting point} for solving the problem corresponding to objective function $\rho_1$. If we are successful, we may then use the solution, $\mu^1$, as our starting point for finding $\mu^2$. Continuing in this fashion we aim to solve a sequence of problems where the final problem corresponds to the function with which we are concerned. The solution, $\mu^\upsilon$, is the solution we seek.

This may be thought of as a homotopy method, in the sense that we solve a sequence of problems corresponding to a sequence of perturbed linear constraints $b^0,b^1,b^2,\dots$ where the solution corresponding to the constraint $b^0$ is known and the solution corresponding to the constraint $b^{\upsilon}$ is the one we seek. We illustrate in the following example.

\begin{example}[Solving a sequence of minimization problems]\label{ex:5}
	We desire to solve for the constraint $Ax=b=A\rho$ with
	\begin{equation*}
	\rho: s \mapsto \frac{7}{20}+\frac{1}{2}\sin \left(3\pi s^2 \right) - \frac{3}{10} = \frac{1}{20}+\frac{1}{2}\sin \left(3\pi s^2 \right)
	\end{equation*}
	Then, letting $\delta = \frac{1}{10}$, and $\upsilon = 3$, we may consider the sequence of problems corresponding to the objective functions given by
	\begin{equation*}
	\rho_N : s \mapsto \frac{7}{20} + \frac{1}{2}\sin \left(3\pi s^2\right)-N\frac{1}{10}, \;\; N \in \{0,..,3\}
	\end{equation*}
	We computed $\mu^0$ using Newton's method in Example~\ref{ex:2}. Using gradient descent with a step size modifier of $\frac{1}{10}$ and taking $\mu^0$ as our starting value, we obtain $\mu^1$. In the same way, we use $\mu^1$ to find $\mu^2$; finally we use $\mu^2$ to find $\mu^3$. $\mu^3$ is the solution we seek for the minimization problem induced by the objective function $\rho$. The corresponding primal values for various $\lambda$ are shown at in Figure~\ref{fig:ex5lambda}.

	Notice that the translated generating function $\rho$ used to generate the linear constraint in Example~\ref{ex:5} has actually been translated even further than the version used to generate the linear constraint in both Examples~\ref{ex:3} and \ref{ex:4}. This homotopy method appears to also solve the proximal version of the problem from Examples~\ref{ex:3} and \ref{ex:4}.

	\begin{figure}
		\begin{center}
			\includegraphics[width=.4\textwidth]{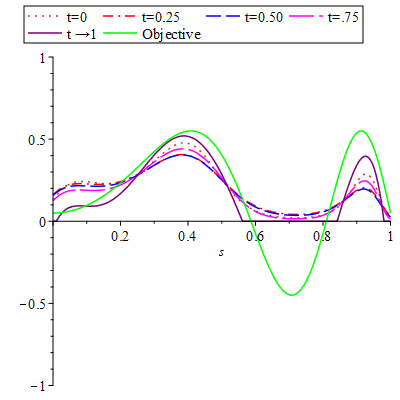}
			\includegraphics[width=.4\textwidth]{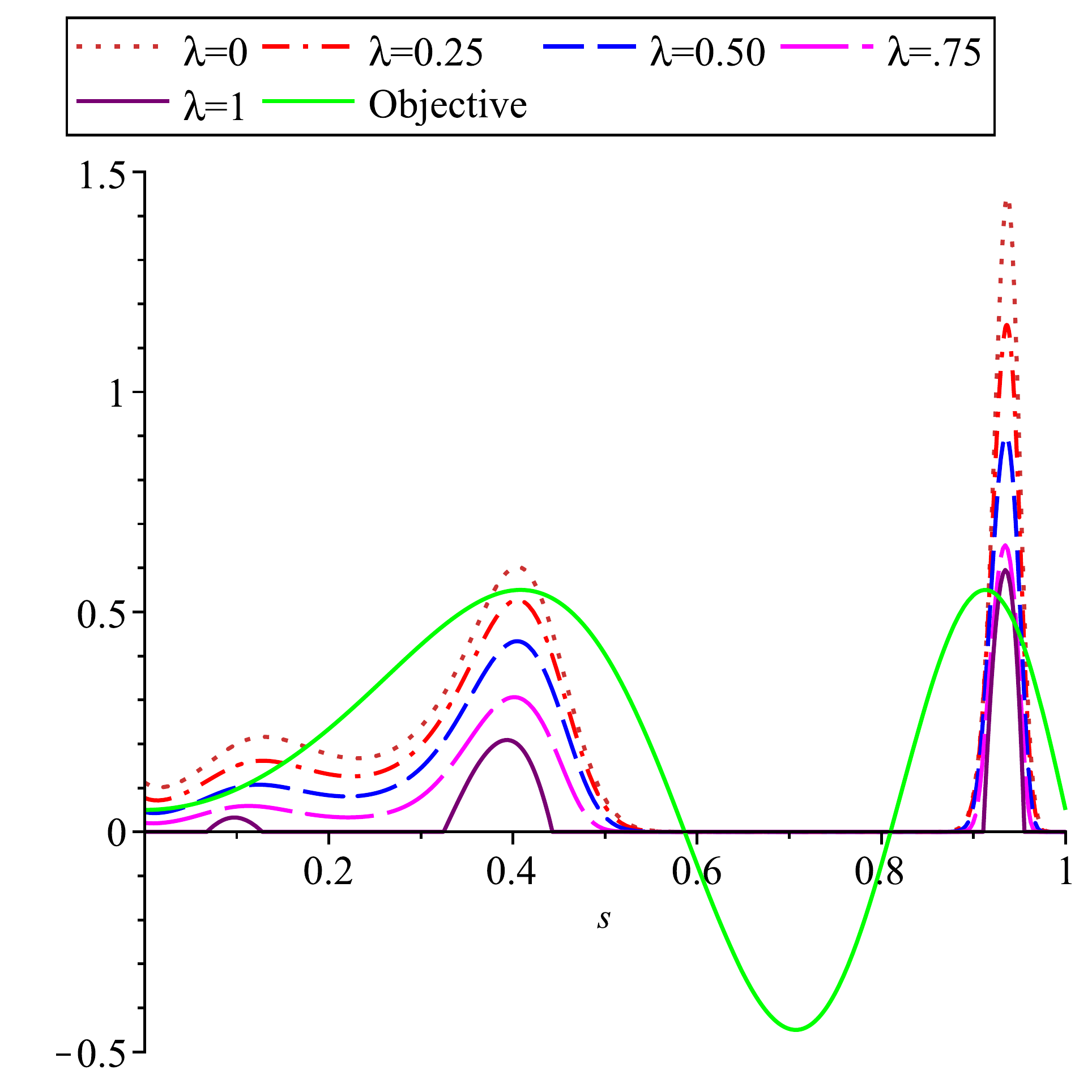}
		\end{center}
		\caption{Primal values obtained by computing with weighted average in Example~\ref{ex:5}.}\label{fig:ex5t}
	\end{figure}

	For comparison, we show the resultant primal values obtained by computing with the weighted average in Figure~\ref{fig:ex5t}. At left we computed with the function $G(\mu)$, and at right we attacked the dual problem directly. The solutions for $f_t$ left in are distinctly different from those at right, which more closely resemble the weighted average solutions in Figure~\ref{fig:ex3} from Example~\ref{ex:3}. In the table below we compare the errors from the linear constraint where $x_t$ is the primal solution obtained by computing with $f_t$. For Example~\ref{ex:3} we used Newton's method. For Example~\ref{ex:5}, we computed $5$ iterates for the first subproblems and $100$ iterates for the final subproblem. When working with $G(\mu)$, we used a gradient descent step size modifier of $1/10$; when attacking the dual problem directly, we used a size of $1$.\\
	\begin{center}
	\begin{tabular}{p{2cm} p{2.5cm} p{2.5cm} p{2.5cm}}
	& Example~\ref{ex:3} & Example~\ref{ex:5} & Example~\ref{ex:5}\\
	& Newton's & $G(\mu)$ & Dual direct\\
	$t$ value & $\|Ax_t -b\|$ & $\|Ax_t -b\|$ & $\|Ax_t -b\|$\\ \hline
	$0$ & $7.46${\scriptsize E}$-11$ & $3.82${\scriptsize E}$-2$ & $7.91${\scriptsize E}$-3$\\
	$0.25$ & $6.26${\scriptsize E}$-11$ & $3.61${\scriptsize E}$-2$  & $3.37${\scriptsize E}$-2$\\
	$0.5$ & $2.09${\scriptsize E}$-10$ & $3.55${\scriptsize E}$-2$ & $7.33${\scriptsize E}$-2$\\
	$0.75$ & $4.42${\scriptsize E}$-10$ & $3.24${\scriptsize E}$-2$ & $1.25${\scriptsize E}$-1$\\
	$\rightarrow 1$ & $2.61${\scriptsize E}$-3$ & $2.04${\scriptsize E}$-2$ & $1.72${\scriptsize E}$-1$
	\end{tabular}
\end{center}
	Computing with the weighted average, we have solutions which do a poorer job of satisfying the linear constraint than in Example~\ref{ex:3}, where $\rho$ has been translated downward by a smaller amount. While the observations we will make about the proximal case below suggest that a better satisfaction of the linear constraint may be possible if we continue to run more iterates, it is also likely that we have reached the limitations of what data the weighted average can be successfully used for. The reasons are as follows.

	Since $\rho$ returns negative values, $\rho \notin {\rm dom}(I_{f_t})$. For this reason, it is difficult to verify whether or not the conditions for strong duality hold unless we can find some other $x \in {\rm dom}I_{f_t}$ such that $Ax=b$ (for example, our numerically obtained solutions for $t<1$ from Example~\ref{ex:3}).

	In fact, $\rho$ may have been translated so far downward that it may no longer be possible to satisfy the linear constraint. This occurs if there does not exist an $x \in {\rm dom}(I_{f_t})$ such that $Ax=b$. Since $b$ still lies in the positive orthant, it is difficult to verify whether this has occurred for the present example. However, further translations downward will eventually yield a data vector $b$ which does not lie in the non-negative orthant. Since the monomials $a_1(s),\dots,a_n(s)$ are non-negative on $[0,1]$, $Ax$ may only lie outside of the non-negative orthant if $x(s)$ takes on negative values in $[0,1]$. However, such an $x(s)$ is not in the domain of $I_{f_t}$ unless $t=1$. In such a case, the linear constraint cannot possibly be satisfied.

	In other words, if $\rho(s) \in L^1([0,1])$ is non-negative, the constraint definitely can be satisfied (indeed, it is satisfied by $\rho$). If $A\rho$ lies outside of the non-negative orthant, the constraint definitely cannot be satisfied. If $\rho(s)$ takes on negative values in $[0,1]$ but $A\rho$ is still in the positive orthant, determining whether or not the linear constraint can be satisfied may be more difficult.

	From a numerical standpoint, we may attempt to check by taking the linear system $Mx=b$ --- where $M$ is the $({\rm \# moments})\times ({\rm \# abscissas})$ matrix representing a discretization of $A$, where cell $i$ in a row $j$ consists of the $i$th weight multiplied by the values of $a_j$ evaluated at the $i$th abscissas --- for $x$ with the requirement that $x$ lie in the positive orthant. Decreasing the number of abscissas to match the number of moments eliminates free variables, although we pay the price of having possibly eliminated some feasible solutions (solutions lying in the positive orthant). With $8$ moments and $8$ abscissas, the unique solution $x$ for Example~\ref{ex:2} lies in the positive orthant. For Example~\ref{ex:3}, $x$ lies just outside of the positive orthant, but the unit precision distance we obtained from the linear constraint with Newton's method indicates that by increasing the number of moments to $20$ we have recovered a feasible solution. For Example~\ref{ex:5} with $8$ moments and $8$ abscissas, our uniquely determined $x$ lies twice as far from the positive orthant. Thus we have a certificate of feasibility for Example~\ref{ex:3}, experimental evidence of feasibility for Example~\ref{ex:4}, and reasonable doubt that feasibility is possible for Example~\ref{ex:5}.
	
	\begin{figure}
		\begin{center}
			\includegraphics[width=.4\textwidth]{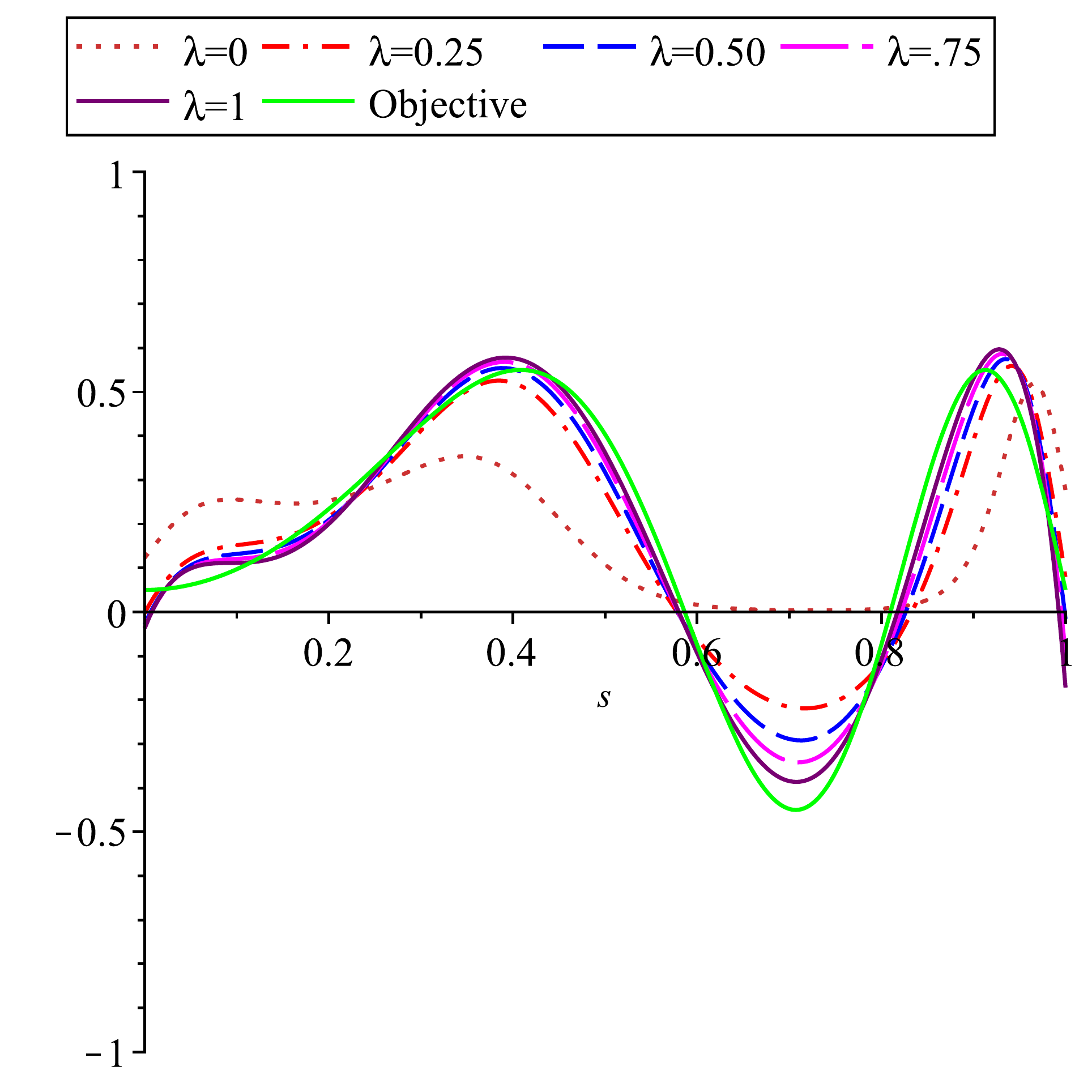}
		\end{center}
		\caption{Example~\ref{ex:5} demonstrates that solving a sequence of problems makes more solutions accessible.}\label{fig:ex5lambda}
	\end{figure}

	The proximal average, by contrast, does not entail such theoretical problems. After running the first subproblems to $5$ iterates, with a gradient descent step size modifier of $1/10$ for minimizing $G(\mu)$ we record the errors from the linear constraint after varying numbers of iterates for the final subproblem as follows.
	\begin{center}
		\begin{tabular}{p{2cm} p{2.5cm} p{2.5cm} p{2.5cm}}
		 & $100$ iterates  & $1100$ iterates & $2100$ iterates\\
			 $\lambda$ value & $\|Ax_\lambda -b\|$ & $\|Ax_\lambda -b\|$ & $\|Ax_\lambda -b\|$ \\ \hline
			$0$ & $3.82${\scriptsize E}$-2$ & $1.96${\scriptsize E}$-2$ & $1.34${\scriptsize E}$-2$ \\
			$0.25$ & $2.97${\scriptsize E}$-2$ & $4.43${\scriptsize E}$-3$ & $4.35${\scriptsize E}$-3$ \\
			$0.5$ & $1.85${\scriptsize E}$-2$ & $1.77${\scriptsize E}$-3$ & $1.72${\scriptsize E}$-3$ \\
			$0.75$ & $1.12${\scriptsize E}$-2$ & $6.75${\scriptsize E}$-4$ & $6.53${\scriptsize E}$-4$ \\
			$1$ & $8.85${\scriptsize E}$-3$ & $1.90${\scriptsize E}$-3$ & $1.66${\scriptsize E}$-3$ 
		\end{tabular}
	\end{center}
	For $\lambda > 0$, conditions for strong duality are still satisfied by $\rho$, and the problem is still feasible. This, combined with the apparent visual fit for $\lambda=0.25,0.5,0.75,1$, suggests that the homotopy method is working, albeit slowly.

	When we attack the dual problem directly, the performance improves. We find that we are able to obtain solutions with only two subproblems, solving first with $N=0$ and then with $N=3$. After solving the $N=0$ case with Newton's method, we record the errors from the linear constraint after varying numbers of iterates of gradient descent (with no step size modification) for the second subproblem as follows.
	\begin{center}
		\begin{tabular}{p{2cm} p{2.5cm} p{2.5cm} p{2.5cm}}
			& $100$ iterates  & $1100$ iterates & $2100$ iterates \\
			$\lambda$ value & $\|Ax_\lambda -b\|$ & $\|Ax_\lambda -b\|$ & $\|Ax_\lambda -b\|$ \\ \hline
			 $0$ & $9.12${\scriptsize E}$-3$ & $4.01${\scriptsize E}$-3$ & $3.37${\scriptsize E}$-3$\\
			 $0.25$ & $2.35${\scriptsize E}$-3$ & $9.95${\scriptsize E}$-4$ & $5.59${\scriptsize E}$-4$\\
			 $0.5$ & $1.07${\scriptsize E}$-3$ & $4.00${\scriptsize E}$-4$ & $2.03${\scriptsize E}$-4$\\
			 $0.75$ &$4.79${\scriptsize E}$-4$ & $1.49${\scriptsize E}$-4$ & $7.89${\scriptsize E}$-5$\\
			 $1$ &$1.57${\scriptsize E}$-4$ & $7.58${\scriptsize E}$-6$ & $1.77${\scriptsize E}$-6$
		\end{tabular}
	\end{center}

	The apparent necessity of homotopy methods when computing with proximal averages when $\rho$ returns lower negative values, particularly for $\lambda$ nearer to $0$, may be related to the penalty for negative values becoming more and more extreme as $\lambda \rightarrow 0$, finally achieving a hard barrier at $\lambda =0$.
\end{example}

\section{Conclusion}\label{s:conclusion}

In this paper, we have catalogued advantages and disadvantages of computing with entropy functionals constructed from proximal averages instead of weighted averages. The weighted average affords ease of computation with hard barriers, but fewer problems may be solvable. In contrast, the proximal average allows us to choose graphically a choice from the net of primal solutions which may afford a better visual fit by being flexible with the enforcement of the barrier. We have explained from a theoretical standpoint why this is the case, and have illustrated it in practice with our examples, giving special attention to the computational challenges one may encounter when working with steep penalties. We have also shown how the Lambert $\W$ function is instrumental in both the weighted averages and proximal averages. In so doing, we have shown how the human-machine collaboration so frequently championed by Borwein may be used to compute hard proximal averages.

We suggest several possibilities for continued investigation.

\begin{enumerate}[label=(\roman*)]
\item It is natural to consider also proximal averages employing for the Fermi--Dirac entropy, which admits hard barriers on \emph{both} sides of a closed interval $[0,1]$. The net produced by the proximal average of the Fermi--Dirac entropy with the Boltzmann--Shannon entropy should admit a hard barrier against negative numbers and a flexible barrier against numbers greater than $1$. One could also consider the net produced by the Fermi--Dirac entropy with the energy.

\item One may also consider the proximal average of two log barriers with empty intersection of their domains.

\item It is quite natural to investigate the case where one replaces the energy (as the proximal term in the construction of the proximal average) with another supercoercive function.

\item Another natural question is: what might we say about the epigraphs of the net of primal solutions for the entropy minimization problem when proximal averages are employed? May we obtain results on some form of continuous transformation of the primal solutions?
\end{enumerate}

Such investigations are likely to prove interesting, and will almost certainly demand the use of similar human-machine collaboration techniques. This present work is a step in that direction and is a natural template for such future investigation. We conclude by noting that the visualization of the entire family $f_\lambda$ of functions admitted by the proximal average illustrate epi-continuity in a beautiful and natural way.

\subsection*{Dedication}

This paper is dedicated to the fond memory of Jonathan M. Borwein, our adviser, mentor, and friend. Jon's guiding philosophy and inspiration underpin not only this work, but so much of everything we do --- and who we are are --- as mathematicians.

\end{document}